\documentclass[10pt]{amsart} 
\usepackage[margin=3cm]{geometry}
\usepackage{hyperref}
\usepackage{pb-diagram}
\usepackage{amsmath, amssymb, amsthm}
\usepackage{tikz}
\usepackage{extpfeil}
\usetikzlibrary{matrix,arrows,decorations.pathmorphing}
\usepackage{paralist}
\usepackage{color}

\newtheorem{thm}{Theorem}

\newtheorem{prp}[thm]{Proposition}

\theoremstyle{definition}
\newtheorem{df}[thm]{Definition}
\newtheorem{exa}[thm]{Example}
\newtheorem{cor}[thm]{Corollary}

\theoremstyle{remark}
\newtheorem{rem}[thm]{Remark}

\numberwithin{equation}{section}
\numberwithin{thm}{section}

\DeclareMathOperator{\id}{id}

\DeclareMathOperator{\map}{map}

\DeclareMathOperator{\Sur}{Sur}

\DeclareMathOperator{\St}{St}
\DeclareMathOperator{\sh}{sh}
\DeclareMathOperator{\card}{card}

\def\Sp{\mathbf{Top}}

\def\hTop{\mathbf{hTop}}
\def\Top{\mathbf{Top}}
\def\Set{\mathbf{Set}}

\def\vP{\vec{P}}
\def\vR{\vec{\R}}

\def\vT{\vec{T}}
\def\vpi{\vec{\pi}}
\def\vI{\vec{I}}
\def\vSigma{\vec{\Sigma}}

\def\R{\mathbb{R}}
\def\Z{\mathbb{Z}}

\def\Set{\mathbf{Set}}
\def\Cat{\mathbf{Cat}}

\def\fF{\mathfrak{F}}
\def\fH{\mathfrak{H}}

\def\cA{\mathcal{A}}
\def\cB{\mathcal{B}}
\def\cC{\mathcal{C}}

\def\cE{\mathcal{E}}

\def\cM{\mathcal{M}}

\def\cS{\mathcal{S}}
\def\cT{\mathcal{T}}
\def\cU{\mathcal{U}}

\def\bO{\mathbf{0}}
\def\bI{\mathbf{1}}

\def\bx{\mathbf{x}}
\def\by{\mathbf{y}}
\def\ba{\mathbf{a}}
\def\bb{\mathbf{b}}

\def\vSP{\mathrm{S}\vec{\mathrm{P}}}

\mathchardef\mhyphen="2D

\title{Stable components of directed spaces}
\author{Krzysztof Ziemia\'nski}

\thanks{ Krzysztof Ziemia\'nski \\
	\textbf{Address: }              
              Institute of Mathematics, Polish Academy of Sciences,
               ul. \'Sniadeckich 8, 00-656 Warszawa, Poland. \\
	\textbf{E-mail: } ziemians@mimuw.edu.pl \\
	\textbf{ORCID: } 0000-0001-7695-4028
}

\begin{document}

%

\begin{abstract}
	In this paper, we introduce the notions of stable future, past and total component systems on a directed space with no loops. Then, we associate the stable component category to a stable (future, past or total) component system. Stable component categories are enriched in some monoidal category, eg. the homotopy category of spaces, and  carry information about the spaces of directed paths between particular points. It is shown that the geometric realizations of finite pre-cubical sets with no loops admit unique minimal stable (future/past/total) component systems. These constructions provide a new family of invariants for directed spaces.
\end{abstract}

\maketitle

\section{Introduction}

Directed spaces, or d-spaces, \cite{Gr} are topological spaces with a distinguished family of paths, called directed paths. They can be used for modeling the behavior of concurrent programs. Points of the directed space represent possible states of a concurrent program, while directed paths represent possible partial executions. This approach allows to employ topological invariants of d-spaces to examine executions of concurrent programs which they represent. All d-spaces considered in this paper are assumed to have no loops, i.e., all directed loops are constant. d-spaces with no loops are slightly more general than partially ordered spaces (po-spaces). We will write $x\leq y$ if there exists a d-path from $x$ to $y$.

Unfortunately, most of the classical (non-directed) homotopy invariants do not have satisfactory directed counterparts. In \cite{Gr}, Grandis introduces the directed fundamental category $\vpi_1(X)$ \cite{Gr} of a d-space $X$. The objects of $\vpi_1(X)$ are points of $X$ and the morphisms from $x$ to $y$ are homotopy classes of d-paths from $x$ to $y$. This is a directed analogue of the fundamental groupoid of a topological space. Alas, unlike for fundamental groupoids, directed fundamental categories are not naturally equivalent to any finite, or even countable, category, except in the most trivial cases.

To overcome this problem, several authors introduced and studied component categories \cite{FGHR, GHComp2,RInv}, which are quotient categories of $\vpi_1(X)$ (or some related categories as in \cite{RInv}). These papers share a common idea, which will be recalled shortly below on the example of \cite{GHComp2}. For a po-space $X$, consider a class $\cE$ of morphisms of $\vpi_1(X)$ that are ``equivalences"; namely, this class consists of \emph{Yoneda morphisms}.  A morphism $\sigma\in\cC(x,y)$ of a category $\cC$ is a Yoneda morphism if it induces bijections:
\begin{itemize}
\item[(P)]{$\cC(z,x)\ni\alpha \to \sigma\circ \alpha \in \cC(z,y)$, whenever $\cC(z,x)\neq \emptyset$, and }
\item[(F)]{$\cC(y,z)\ni \alpha \to \alpha\circ\sigma\in \cC(x,z)$, whenever $\cC(y,z)\neq \emptyset$.}
\end{itemize}
Then the authors prove that there is a unique maximal system $\Sigma\subseteq \cE$ with good categorical properties and define the component category of $X$ as the category of fractions $\vpi_1(X)[\Sigma^{-1}]$. Informally, condition (P) can be stated as ``$\sigma$ is a past equivalence between $x$ and $y$ with respect to $z$, for every $z$ such $x$ that is reachable from $z$".

This approach works very well in many cases. The main motivation for searching for a new definition of the component system is the following example, due to Dubut \cite{DubutPhD}, for which this approach fails, i.e., the component category obtained is not finite. Consider the d-space that is the geometric realization of the precubical set $X$
\begin{equation}
		\begin{tikzpicture}
		\def\x{2};
		\fill[color=black!6!white] (0,0)--(3*\x,0)--(3*\x,\x)--(\x,\x)--(\x,2*\x)--(0,2*\x);
		\draw[thick] (0,0)--(3*\x,0)--(3*\x,\x)--(0,\x);
		\draw[thick] (0,0)--(0,2*\x)--(\x,2*\x)--(\x,0);
		\draw[thick] (2*\x,0)--(2*\x,\x);
		\node[above] at (\x/2,2*\x) {$e$};
		\node[below] at (2.5*\x,0) {$e$};
		\draw[thick,->] (0,0)--(0,0.5*\x);
		\draw[thick,->] (\x,0)--(\x,0.5*\x);
		\draw[thick,->] (2*\x,0)--(2*\x,0.5*\x);
		\draw[thick,->] (3*\x,0)--(3*\x,0.5*\x);
		\draw[thick,->] (0,\x)--(0,1.5*\x);
		\draw[thick,->] (\x,\x)--(\x,1.5*\x);
		\draw[thick,->] (0*\x,0)--(0.5*\x,0);
		\draw[thick,->] (1*\x,0)--(1.5*\x,0);
		\draw[thick,->] (2*\x,0)--(2.5*\x,0);
		\draw[thick,->] (0*\x,\x)--(0.5*\x,\x);
		\draw[thick,->] (1*\x,\x)--(1.5*\x,\x);
		\draw[thick,->] (2*\x,\x)--(2.5*\x,\x);
		\draw[thick,->] (0,2*\x)--(0.5*\x,2*\x);
		\node at (\x/2,\x/2){$A$};
		\node at (\x/2,\x*1.5){$B$};
		\node at (\x*1.5,\x/2){$C$};
		\node at (\x*2.5,\x/2){$D$};
        \fill(0.2*\x, 0.1*\x) circle (0.05);
        \fill(0.7*\x, 0.3*\x) circle (0.05);
        \node[above left] at (0.2*\x, 0.1*\x) {x};
        \node[above right] at (0.7*\x, 0.3*\x) {y};
        \fill(2.4*\x, 0.8*\x) circle (0.05);
        \node[right] at (2.4*\x, 0.8*\x) {z};
	\end{tikzpicture}
\end{equation}
where the edges marked by $e$ are identified. The areas $A$ and $D$ are closed.
 In this case, no non-identity morphism of $\vpi_1(X)$ represented by a path in the square $A$ is a Yoneda morphism; for any two points $x\lneq y\in A$ one can find a point $z\in D$ such that there are two morphisms from $x$ to $z$ but only one from $y$ to $z$. As a consequence, no two points in $A$ are Yoneda equivalent, and the component category of $X$ is uncountable.

In this paper we propose a new definition of the component category of a directed space with no loops. These categories, called stable component categories, are finite for all d-spaces which are geometric realizations of finite precubical sets with no loops. The main idea is to relax conditions (P) and (F) defining when points $x$ and $y$ are considered equivalent. Note that, in the example above, for any $x\lneq y\in A$ we have a bijection $\vpi_1(X)(y,z)\simeq\vpi_1(X)(x,z)$ if $z\in D$ is "large enough", i.e., if $z\geq z(D,x,y)$ for some $z(D,x,y)$ depending on $D$, $x$ and $y$. Thus, we no longer require that $x$ and $y$ are future equivalent with respect to every $z\geq y$ but that they are equivalent with respect to "large enough" points in every component. This leads to an axiomatic definition of component systems (Definition \ref{d:StableComponentSystem}); a decomposition of $X=\coprod A_i$ into disjoint subsets is a stable future component system if every pair $(A_i,A_j)$ is future stable, i.e., satisfies, among others, a condition similar to the property of sets $A$ and $D$ formulated above. In a similar way, we define past stable component systems, and total stable component systems, which are both past and future stable. In the example above, $X=A\mathop{\dot\cup} B\mathop{\dot\cup}  C\mathop{\dot\cup}  D$ is a stable total component system on $X$; see Example \ref{x:Dubut} for details.
A similar approach is presented in \cite{GHK}.

Stable component systems would not allow us to define invariants of d-spaces if one would not be able to single out a unique stable component system for a given d-space. It is not true in general that every d-space without loops admits a unique coarsest stable component system (Remark \ref{r:Counterexample}). Nevertheless, we are able to prove that if $X$ admits any finite component system, then there exists the unique coarsest one (Theorem \ref{t:UniqueSystem}). The class of d-spaces having a finite (and, therefore, a coarsest) component system includes the geometric realizations of finite cubical complexes with no loops (Theorem \ref{t:InteriorSystem}). Eventually, we define the component category associated to a stable component system (Section \ref{s:ComponentCategory}). These results allow to define three categories for a sufficiently good d-space $X$: the component categories of its coarsest stable future, past and total component system.

Apart from three possible flavors (future, past and total), the stable components systems in this paper are parametrized by some class of equivalences of topological spaces. There are many possible choices for when a d-path $\alpha$ from $x$ to $y$  should be regarded a future equivalence with respect to $z$. One of the possible choices is to claim that $\alpha$ is an equivalence if
\begin{equation}\label{e:ZeroEq}
	\vpi_1(X)(y,z)\ni [\omega] \mapsto [\alpha*\omega]\in \vpi_1(X)(x,z)
\end{equation}
is a bijection, as in the example above; another possibility is to require that the concatenation map
\begin{equation}\label{e:InfEq}
	\vP_X(y,z) \ni \omega \mapsto \alpha*\omega\in \vP_X(x,z)
\end{equation}
between the spaces of directed paths with the given endpoints is a weak homotopy equivalence. These are the two most natural, and extreme, cases but other choices  of the class of equivalences are possible, like all homology equivalences, all maps inducing isomorphisms on homotopy groups up to some dimension, etc. To handle all these cases simultaneously, we introduce a class $\fF$ of equivalences in the category of topological spaces $\Top$. The most important examples are the class $\fF_\infty$ of all weak homotopy equivalences, as in (\ref{e:InfEq}) and the class $\fF_0$ of maps inducing a bijection between path-connected components (\ref{e:ZeroEq}).

The component categories of stable $\fF$--component systems will be enriched in some monoidal category $\cS$, assuming that there is given a monoidal functor $L:\Top\to \cS$ that sends maps belonging to $\fF$ into isomorphisms in $\cS$. Thus, for a d-space $X$ admitting a coarsest component systems we obtain three families of $\cS$--enriched categories, depending on the choice of $\fF$ and $L$: the future coarsest $\fF$--component category $\vSP^+_{\fF,L}(X)$, the past one $\vSP^-_{\fF,L}(X)$ and the total one $\vSP^\pm_{\fF,L}(X)$.

 Three choices of $\fF$ and $L$ deserve a special attention:
\begin{itemize}
\item{
	$\fF=\fF_0$ is the class of maps inducing an isomorphism on $\pi_0$, and $L=\pi_0:\Top\to\Set$. Then the component category $\vSP_{\fF,L}(X,\cA)$ of a component system $\cA$ is $\Set$--enriched, i.e., it is a category in the usual sense. If $\cA$ is the coarsest component system, $\vSP^\mu_{\fF_0,\pi_0}(X;\cA)$, $\mu\in\{+,-,\pm\}$, can be regarded as directed analogues of $\pi_0(X)$.
}
\item{
	$\fF=\fF_\infty$ is the class of weak homotopy equivalences and $L:\Top\to\hTop$ is the forgetful functor into the homotopy category. The component category is enriched in the homotopy category and carries information about the homotopy types of the spaces appearing as directed path spaces in $X$. This can be viewed as ``the directed total homotopy group".
}
\item{
	$\fF=\fH_R$ is the class of maps inducing an isomorphism on $H_*(-;R)$ for some principal ideal domain $R$, and $L=H_*$  is the homology functor into the category of graded $R$--modules with the monoidal structure given by the graded tensor product. This choice seems to be best suited for specific calculations.
}
\end{itemize}

\section{Preliminaries}

\subsection*{Directed spaces}

For a topological space $X$, let $P_X=\map([0,1],X)$ denote the space of paths on $X$ with the compact-open topology.

\begin{df}\label{d:dSpace}
    \emph{A directed space}, or \emph{a d-space} \cite{Gr} is a pair $(X,\vP_X)$, where $X$ is a topological space and $\vP_X\subseteq P_X$ is a family of paths, called \emph{directed paths} or \emph{d-paths}, that satisfies the following conditions:
     \begin{enumerate}[\normalfont (a)]
     \item{Every constant path $\sigma_x$, $x\in X$, is a d-path.}
     \item{The concatenation of two d-paths is a d-path. Namely, if $\alpha,\beta\in \vP_X$ and $\alpha(1)=\beta(0)$, then $\alpha*\beta\in\vP_X$, where
     \[
        \alpha*\beta(t)=
        \begin{cases}
            \alpha(2t) & \text{for $0\leq t \leq \tfrac{1}{2}$,}\\
            \beta(2t-1) & \text{for $\tfrac{1}{2}\leq t \leq 1$.}
        \end{cases}
     \]
     }
     \item{Non-decreasing reparametrizations of d-paths are d-paths, i.e., if $\alpha\in \vP_X$ and $f:[0,1]\to[0,1]$ is a non-decreasing continuous function, then $\alpha\circ f\in \vP_X$.}
     \end{enumerate}
     For short, we will write $X$ instead of $(X,\vP_X)$. The subspace $\vP_X\subseteq P_X$ will called \emph{the d-structure} of $X$.
\end{df}

For d-spaces $X,Y$, a continuous map $f:X\to Y$ is \emph{a d-map} if $f\circ \alpha\in \vP_Y$ for every d-path $\alpha\in\vP_X$. Equivalently, one can require that the image of $\vP_X$ under the map $P_X\to P_Y$ induced by $f$ is contained in $\vP_Y$. A d-map is \emph{a d-homeomorphism} if it is a bijection and its inverse is a d-map.

Some examples of d-spaces are described below:

\begin{exa}
	\emph{The directed real line} $\vR=(\R,\vP_{\vR^n})$, where $\vP_{\vR^n}$ is the subspace of non-decreasing continuous functions $[0,1]\to\R$.
\end{exa}

\begin{exa}
	\emph{The product $X\times Y$} of d-spaces $X$ and $Y$ is defined by
	\begin{equation*}
		\vP_{X\times Y}= \vP_X\times \vP_Y = \{\alpha=(\alpha_X,\alpha_Y)\;|\; \alpha_X\in\vP_X,\; \alpha_Y\in\vP_Y\}.
	\end{equation*}
	An important example is \emph{the directed Euclidean space $\vR^n=\overbrace{\vR\times\dots\times\vR}^{\text{$n$ times}}$}.
\end{exa}

\begin{exa}
	For a d-space $X$, every subspace $A\subseteq X$ has the inherited d-structure given by $\vP_A:=\vP_X\cap P_A$. In particular, this defines:
	\begin{itemize}
	\item{\emph{The directed interval $\vI=[0,1]\subseteq \vR$,}}
	\item{\emph{The directed cube $\vI^n=[0,1]^n\subseteq \vR^n$}.}
	\item{\emph{The boundary of $\vI^n$},
	\begin{equation*}
		\partial\vI^n = \{(x_1,\dots,x_n)\in \vI^n\; |\; \text{at least one of $x_i$'s equals $0$ or $1$}\}\subseteq \vI^n.
	\end{equation*}}
	\end{itemize}
\end{exa}

\begin{exa}
	\emph{The opposite d-space} of $X=(X,\vP_X)$ is the d-space
	\begin{equation*}
		X^{op}=(X,\vP^{op}_X=\{\alpha^{op}\;|\;\alpha\in \vP_X\}),
	\end{equation*}
	where $\alpha^{op}(t)=\alpha(1-t)$.
\end{exa}

Fix a d-space $X$.
For subsets $A,B,C\subseteq X$, denote
\begin{equation}
	\vP_A(B,C)=\{\alpha\in \vP_A\;|\; \alpha(0)\in B,\; \alpha(1)\in C\}.
\end{equation}
A pair of d-paths $\alpha,\beta\in \vP_X$ induces the map
\begin{equation}
	\vP(\alpha,\beta):\vP_X(\alpha(1),\beta(0))\ni \omega \mapsto \alpha*\omega*\beta \in \vP_X(\alpha(0),\beta(1)).
\end{equation}

The monoid $\Sur^+([0,1])$ of non-decreasing surjective self-maps of $[0,1]$ acts naturally on $\vP_X$. The quotient space of this action is \emph{the trace space $\vT_X$} of $X$, and elements of $\vT_X$ will be called \emph{traces}; as above, $\vT_X(A,B)=\vP_X(A,B)/\Sur^+([0,1])$ stands for the space of traces having the required endpoints.
As shown in \cite{RInv}, for any points $x,y\in X$ the quotient map $\vP_X(x,y)\to \vT_X(x,y)$ is a weak homotopy equivalence. As a result, in the considerations below, path spaces can be replaced by trace spaces without any consequences.

A d-space $X$ \emph{has no loops} if $\vP_X(x,x)=\{\sigma_x\}$ for every $x\in X$, where $\sigma_x$ stands for the constant path. The relation
\begin{equation}
	x\leq y \; \Leftrightarrow \; \vP_X(x,y)\neq\emptyset
\end{equation}
is reflexive and transitive for every d-space $X$. If $X$ is has no loops, it is also antisymmetric and then $(X,\leq)$ is a partial order. For a subset $A\subseteq X$ and $x\in X$, we denote
\begin{equation}
	A_{\leq x}:=\{y\in A\;|\; y\leq x\},\qquad A_{\geq x}:=\{y\in A\;|\; y\geq x\}.
\end{equation}
The d-space $X$ with the relation $\leq$ is not necessarily a partially ordered space in the sense of Nachbin \cite{Nachbin} because this partial order is not, in general, closed, regarded as a subset of $X\times X$.

\subsection*{Equivalences}
Let $\Top$ denote the category of topological spaces and continuous maps.
\begin{df}
    A family $\fF$ of morphisms of $\Top$ is \emph{an equivalence system} if the following conditions are satisfied:
    \begin{enumerate}[\normalfont (a)]
        \item{$\fF$ satisfies the 2--out--of--3 property.}\label{i:FTT}
        \item{If maps $f,g$ are homotopic, then either $f,g\in\fF$ or $f,g\not\in \fF$.}
        \item{$\fF$ contains all weak homotopy equivalences.} \label{i:FWHE}
        \item{If $f\in \fF$, then $\pi_0(f)$ is a bijection.}\label{i:FPiZero}
        \item{For a finite family of maps $\{f_i:X_i\to Y_i\}_{i\in I}$, the disjoint sum $\coprod_i f_i:\coprod_i X_i\to \coprod_i Y_i$ belongs to $\fF$ if and only if $f_i\in \fF$ for all $i\in I$.}
        \item{If $f:X\to Y,g:X'\to Y'\in \fF$, then $f\times g:X\times X'\to Y\times Y'\in \fF$.}
    \end{enumerate}
    The morphisms belonging to $\fF$ will be called \emph{$\fF$--equivalences}. Two spaces $X,Y$ are \emph{$\fF$--equivalent} if they can be connected by a zig-zag of $\fF$--equivalences. A space $X$ is \emph{$\fF$--contractible} if $X\to \{*\}$ is an $\fF$--equivalence; clearly any space that is $\fF$--equivalent to a contractible space is $\fF$--contractible.
\end{df}

\begin{exa}
    For $k\in \Z_{\geq 0}\cup \{\infty\}$, let $\fF_k$ be the family of maps that induce an isomorphism on $\pi_0$ and isomophisms on all homotopy groups $\pi_n$, for every $n\leq k$ and for every choice of basepoints. Then $\fF_k$ is an equivalence system. This includes the examples mentioned in the introduction: $\fF_0$ is the class of maps inducing isomorphism on $\pi_0$, and $\fF_\infty$ is the class of weak homotopy equivalences.
\end{exa}

Notice that conditions (\ref{i:FWHE}) and (\ref{i:FPiZero}) guarantee that
\begin{equation}
    \fF_\infty \subseteq \fF \subseteq \fF_0
\end{equation}
for every equivalence system $\fF$.

\begin{exa}
    For an abelian group $A$ and $k\in \Z_{\geq 0}\cup \{\infty\}$, the family $\fH_{A,k}$ of maps that induce an isomorphism on $H_n(-;A)$ for $n\leq k$ is an equivalence system.
\end{exa}

In the remaining part of the paper we will assume that $\fF$ is a fixed equivalence system.

\section{Stable components}

Fix a d-space $X$ with no loops and an equivalence system $\fF$.

\begin{df}
	A subset $A\subseteq X$ is
	\begin{enumerate}[\normalfont (a)]
	\item{\emph{d-convex} if $\vP_A(A,A)=\vP_X(A,A)$, i.e., every directed path having endpoints in $A$ is contained in $A$.}
	\item{\emph{future connected} if, for every $x,y\in A$, there exists $z\in A$ such that $\vP_A(x,z)\neq \emptyset\neq \vP_A(y,z)$.}
	\item{\emph{past connected} if, for every $x,y\in A$, there exists $z\in A$ such that $\vP_A(z,x)\neq \emptyset\neq \vP_A(z,y)$.}
	\end{enumerate}
\end{df}

\begin{rem}
    A future connected subset may or may not have a maximal point. If such a point exists, it is unique.
\end{rem}

\begin{df}\label{d:Cofinal}
	We say that future (resp. past) connected subsets $A,B\subseteq X$ are \emph{cofinal} (resp. \emph{coinitial}) if there exists $y\in A\cap B$ such that $A_{\geq y}=B_{\geq y}$ (resp. $A_{\leq y}=B_{\leq y}$).
\end{df}

From this point we will introduce only the future versions of definitions;  the past counterparts are the same as the future ones applied to the opposite d-space $X^{op}$.

\begin{df}
    Let $A\subseteq X$ be a future connected subset and let $\alpha\in\vP_X$ be a d-path.
    \begin{enumerate}[\normalfont (a)]
    \item{
        $A$ is \emph{future $\fF$--invariant with respect to $\alpha$} if the map
    	\begin{equation*}
			\vP(\alpha,\beta):\vP_X(\alpha(1),\beta(0))\ni \omega \mapsto \alpha*\omega*\beta \in \vP_X(\alpha(0),\beta(1))	
	   \end{equation*}
    	is an $\fF$--equivalence for every $\beta\in \vP_A$.
    }
    \item{
       A point $z\in A$ \emph{future $\fF$--stabilizes} $\alpha$ in $A$ if $A_{\geq z}$ is future $\fF$--invariant with respect to $\alpha$.
    }
    \item{
        \emph{The future $\fF$--stabilizer of $\alpha$ in $A$} is the set
        \[
            \St^+(\alpha;A)=\St^+_{\fF}(\alpha;A)\subseteq A
        \]
        of all points of $A$ that future $\fF$--stabilize $\alpha$ in $A$. The past $\fF$--stabilizer of $\alpha$ in $A$ will be denoted by $\St^-_\fF(A;\alpha)$.
    }
    \item{
        If $\St^+_\fF(\alpha;A)\neq\emptyset$, then we say that \emph{$\alpha$ future $\fF$--stabilizes in $A$}, or that $A$ is \emph{future $\fF$--stable with respect to $\alpha$}.
    }
    \end{enumerate}
\end{df}

If $A$ has a maximal point $x\in A$, then $\alpha$ future $\fF$--stabilizes in $A$ if and only if the map
\begin{equation}
    \vP_X(\alpha,x):\vP_X(\alpha(1),x)\ni \omega \mapsto \alpha*\omega \in \vP_X(\alpha(0),x)
\end{equation}
is an $\fF$--equivalence.

If $\alpha=\sigma_x$ is a constant path we will sometimes write $\St^+_\fF(x;A)$ instead of $\St^+_\fF(\sigma_x;A)$. The prefixes or indices $\fF$ will be skipped if it is clear which class of equivalences is considered.

Next, we collect elementary properties of stabilizers, which will be used frequently later on.

\begin{prp}\label{p:PropertiesOfStabilizers}
	Fix $\alpha\in \vP_X$ and a future connected subset $A\subseteq X$.
	\begin{enumerate}[\normalfont (a)]
	\item{If $x\in \St^+(\alpha;A)$ and $y\in A_{\geq x}$, then $y\in \St^+(\alpha;A)$. In particular, if $\St^+(\alpha;A)$ is non-empty, then $A$ and $\St^+(\alpha;A)$ are cofinal.}
	\item{$x\in \St^+(\alpha;A)$ if and only if $x\in \St^+(\alpha;A_{\geq x})$.}
	\item{$\St^+(\alpha;A)\subseteq \St(\sigma_{\alpha(t)};A)$ for $t=0,1$.}
	\item{If $\St^+(\alpha;A)\neq \emptyset$ and $x\in A$, then $\St^+(\alpha;A)\cap A_{\geq x}\neq \emptyset$.}\label{i:PCofinal}
	\item{If $\St^+(\alpha_i;A)\neq \emptyset$ for $\alpha_i\in \vP_X$, $i\in [1:n]$ then
	\[
        \St^+(\alpha_1,\dots,\alpha_n; A):=\St^+(\alpha_1;A)\cap\dots\cap  \St^+(\alpha_n;A)\neq \emptyset.
    \]
	}
    \item{If $\beta\in\vP_X$ and $\beta(0)=\alpha(1)$, then
    \[
        \St^+(\alpha,\beta;A)\subseteq \St^+(\alpha*\beta; A).
    \]
    }\label{i:PConcat}
	\item{If $\alpha\sim\beta\in \vP_X(x,y)$, i.e., $\alpha$ and $\beta$ lie in the same path-connected component of $ \vP_X(x,y)$, then  $\St^+(\alpha;A)=\St^+(\beta;A)$.}\label{i:PHtp}
	\end{enumerate}
\end{prp}
\begin{proof}
	(a) and (b) are obvious.
	
	(c): Fix $\beta\in \vP_{A_{\geq x}}$. We have the diagram
	\[
		\begin{diagram}
			\dgARROWLENGTH=3.5em
			\node{\vP(\alpha(1),\beta(0))}
				\arrow{e,t}{\vP(\alpha,\beta(0))}
				\arrow{se,t}{\vP(\alpha,\beta)}
				\arrow{s,l,..}{\vP(\alpha(0),\beta)}
			\node{\vP(\alpha(0),\beta(0))}
				\arrow{s,r,..}{\vP(\alpha(0),\beta)}
		\\
			\node{\vP(\alpha(1),\beta(1))}
				\arrow{e,t}{\vP(\alpha,\beta(1))}
			\node{\vP(\alpha(0),\beta(1))}
		\end{diagram}	
	\]
	which commutes up to homotopy, in which all solid arrows are $\fF$--equivalences. Thus, the 2-out-of-3 property of $\fF$ implies that the dotted ones are also $\fF$--equivalences.
	
	(d) and (e): Since $A$ is future connected, these follow from (a).
	
    (f): Assume that $z\in \St^+(\alpha,\beta; A)$ and $\omega\in\vP_{A_{\geq z}}$. We have
    \[
        \vP(\alpha*\beta,\omega)\sim \vP(\alpha,\omega)\circ \vP(\beta,\sigma_{\omega(0)}).
    \]
    Since $z$ stabilizes both $\alpha$ and $\beta$, both maps in the right-hand composition are $\fF$--equivalences. So the map which is homotopic to their composition also is.

	(g): Homotopic paths induce homotopic maps between the path spaces, and two homotopic maps are either both $\fF$--equivalences or both non--$\fF$--equivalences.
\end{proof}

Notice that (\ref{i:PHtp}) implies that  stability is not a property of a particular path from $x$ to $y$ but rather of its class in $\vpi_1(X)(x,y)$.

\begin{df}
	Let $A,B\subseteq X$ be future connected subsets. The pair $(A,B)$ is \emph{future $\fF$--stable} if $B$ is future $\fF$--stable with respect to  every d-path $\alpha\in\vP_A$.
\end{df}

We do not assume that all d-paths in $A$ are future stabilized by the same element in $B$; it may happen that $(A,B)$ is future stable and yet $\bigcap_{\alpha\in \vP_A} \St^+(\alpha;B)=\emptyset$. However, if $B$ has the final point $y$, then $(A,B)$ is future stable if and only if $y\in\St^+(\alpha;B)$ for all $\alpha\in\vP_A$.

\begin{prp}\label{p:CofinalityPreservesStability}
	Let $A,B,B'\subseteq X$ be future connected subsets. Assume that $B$ and $B'$ are cofinal. Then the pair $(A,B)$ is future stable if and only if $(A,B')$ is  future stable.
\end{prp}
\begin{proof}
	Fix $y$ such that $B_{\geq y}=B'_{\geq y}$ and assume that $(A,B)$ is a stable pair. Then for every $\alpha\in\vP_A$ there exists $x\in \St^+(\alpha;B)_{\geq y}$ (by Proposition \ref{p:PropertiesOfStabilizers}.(d)). By  \ref{p:PropertiesOfStabilizers}.(b) we have $x\in\St^+(\alpha;B_{\geq x})=\St^+(\alpha;B'_{\geq x})$ and then $x\in\St^+(\alpha;B')$.
\end{proof}

\subsection*{Stable d-path spaces}
Let $A,B\subseteq X$ be future connected subsets and assume that the pair $(A,B)$ is future $\fF$--stable.

\begin{prp}\label{p:IndependenceOfStabilizer}
	If $a\in A$ and $z,z'\in\St_\fF^+(\sigma_a,B)$, then the spaces $\vP(a,z)$ and $\vP(a,z')$ are $\fF$--equivalent.
\end{prp}
\begin{proof}
	Since $B$ is future connected, there exist $z''\in B_{\geq z,z'}$, $\beta\in \vP_B(z,z'')$, $\beta'\in \vP_B(z',z'')$. Both maps in the composition
	\[
		\vP_X(a,z)\xrightarrow{\vP(a,\beta)} \vP_X(a,z'')\xleftarrow{\vP(a,\beta')} \vP_X(a,z')
	\]
	are $\fF$--equivalences, since both $z$ and $z'$ stabilize $\sigma_a$ in $B$.
\end{proof}

\begin{prp}\label{p:IndependenceOfPoint}
	Let $a,a'\in A$ and let $z,z'\in B$ be stabilizers of $\sigma_a$ and $\sigma_{a'}$, respectively. Then the spaces $\vP(a,z)$ and $\vP(a',z')$ are $\fF$--equivalent.
\end{prp}
\begin{proof}
	Choose $a''\in A_{\geq a,a'}$, $\alpha\in\vP_A(a,a'')$, $\alpha'\in \vP_A(a',a'')$; these exist since $A$ is future connected.
By \ref{p:PropertiesOfStabilizers}.(d) and (e), there exists $z''\in \St^+(\alpha,\alpha';B)_{\geq z,z'}$. Choose paths $\omega\in\vP_B(z,z'')$ and $\omega'\in\vP_B(z',z'')$. We obtain the following sequence of maps
	\begin{equation}\label{e:EquivMap}
		\vP_X(a,z)\xrightarrow{\vP(a,\omega)}
		\vP_X(a,z'') \xleftarrow{\vP(\alpha,z'')}
		\vP_X(a'',z'')\xrightarrow{\vP(\alpha',z'')}
		\vP_X(a',z'') \xleftarrow{\vP(a',\omega')}
		\vP_X(a',z'),
	\end{equation}
	which are all $\fF$--equivalences.
\end{proof}

As a consequence, we can define \emph{the future stable d-path space} from $A$ to $B$ as
\begin{equation}
	\vSP^+_\fF(A,B):=\vP_X(a,z),
\end{equation}
where $a\in A$ and $z\in \St^+_\fF(\sigma_a,B)$. By Proposition \ref{p:IndependenceOfPoint}, the space $\vSP_\fF^+(A,B)$ is well-defined up to (and only up to) $\fF$--equivalence.

\begin{df}
	A pair $(A,B)$ is \emph{totally $\fF$--stable} if:
	\begin{enumerate} [\normalfont (a)]
		\item{$(A,B)$ is future $\fF$--stable and past $\fF$--stable,}
		\item{There exist $a\in A$ and $b\in B$ such that $b\in\St^+_\fF(a;B)$ and $a\in \St^-_\fF(A;b)$. Such a pair $(a,b)$ will be called \emph{a stabilizing pair} of the totally stable pair $(A,B)$.}
	\end{enumerate}
\end{df}

For a totally stable pair $(A,B)$ and a stabilizing pair $(a\in A,b\in B)$  we have obvious $\fF$--equivalences
\begin{equation}
	\vSP^+_\fF(A,B)\buildrel{\fF}\over\simeq \vP_X(a,b) \buildrel{\fF}\over\simeq \vSP^-_\fF(A,B).
\end{equation}

Next, we formulate some properties of stable pairs. The first proposition is obvious:

\begin{prp}\label{p:PropertiesOfStabilizingPairs}
	Let $(A,B)$ be a totally $\fF$--stable pair and let $(a\in A,b\in B)$ be a stabilizing pair.
	\begin{enumerate}[\normalfont (a)]
	\item{If $\alpha\in \vP_{A_{\leq a}}$ and $\beta\in \vP_{B_{\geq b}}$, then $\vP(\alpha,\beta)$ is an $\fF$--equivalence.}
	\item{If $a'\in A_{\leq a}$, $b'\in B_{\geq b}$, then $(a',b')$ is a stabilizing pair of $(A,B)$.\qed}
	\end{enumerate}	
\end{prp}

\begin{prp}\label{p:CommonStableTraces}
	Let $(A,B)$ and $(A',B')$ be future stable pairs. Assume that $A\cap A'\neq \emptyset$ and that $B$ and $B'$ are cofinal. Then $\vSP^+(A,B)$ and $\vSP^+(A',B')$ are $\fF$--equivalent.
\end{prp}
\begin{proof}
	Choose $a\in A\cap A'$ and $x\in B\cap B'$ such that $B_{\geq x}=B'_{\geq x}$. There exists $y\in \St^+_\fF(\sigma_a;B)_{\geq x}$ (\ref{p:PropertiesOfStabilizers}.(d)) and, by \ref{p:PropertiesOfStabilizers}.(b),
	\[
		y\in \St^+_{\fF}(\sigma_a;B) \Leftrightarrow y\in \St^+_{\fF}(\sigma_a;B_{\geq y}) \Leftrightarrow y\in \St^+_{\fF}(\sigma_a;B').
	\]
	Thus,
	\[
		\vSP^+_\fF(A,B)\overset{\fF}{\simeq} \vP(a,y) \overset{\fF}{\simeq} \vSP^+_\fF(A',B').\qedhere
	\]
\end{proof}

\begin{prp}\label{p:TraceDominates}
	Let $(A,B)$ be a future $\fF$--stable pair. Then the following conditions are equivalent:
	\begin{enumerate}[\normalfont (a)]
	\item{$\vSP^+_\fF(A,B)\neq\emptyset$,}
	\item{$\vP(A,B)\neq \emptyset$,}
	\item{$\vP(x,B)\neq \emptyset$ for every $x\in A$.}
	\end{enumerate}
\end{prp}
\begin{proof}
	If $x\in A$, $y\in B$ and $\vP(x,y)\neq\emptyset$, then every path $\alpha\in\vP(x,y)$ extends to a path $\alpha'\in\vP(x,y')$ for some $y'\in \St^+(\sigma_x;B)_{\geq y}$;  the stabilizer $\St^+(\sigma_x;B)_{\geq y}$ is non-empty by \ref{p:PropertiesOfStabilizers}.(\ref{i:PCofinal}). Thus, $\vSP^+_\fF(A,B)\overset{\fF}{\simeq} \vP(x,y')\neq\emptyset$ and we have (b)$\Rightarrow$(a). Every $x\in A$ admits a point $y\in B$ such that $\vP(x,y)\overset{\fF}{\simeq}\vSP^+_\fF(A,B)$; thus, (a) implies (c). The implication (c)$\Rightarrow$(b) is obvious.
\end{proof}

\subsection*{Stable component systems}

\begin{df}
	A future connected subset $A\subseteq X$ is \emph{future $\fF$--trivial} if the pair $(A,A)$ is future stable and $\vSP^+_\fF(A,A)$ is $\fF$--contractible.
\end{df}

\begin{df}\label{d:StableComponentSystem}
	A family $\cA=\{A_i\}_{i\in I}$ of pairwise disjoint subsets of $X$ such that $\bigcup_{i\in I} A_i=X$ is \emph{a stable future (resp. past) $\fF$--component system} if
	\begin{enumerate}[\normalfont (a)]
	\item{$A_i$ is d-convex for every $i\in I$,}
	\item{$A_i$ is future (resp. past) connected for every $i\in I$,}
	\item{$A_i$ is future (resp. past) $\fF$--trivial for every $i\in I$,}
	\item{the pair $(A_i,A_j)$ is future (resp. past) $\fF$--stable for every $i,j\in I$.}
	\end{enumerate}	
	The stable future (resp. past) $\fF$--component system $\cA$ is \emph{closed} if every $A_i$ has a final (resp. initial) point.
\end{df}
We will occasionally skip the adjective "stable", since this is the only kind of component systems considered in this paper.

\begin{rem}
	In most cases, the assumption that $\vSP^+_\fF(A,A)$ is $\fF$--contractible is obsolete. But not always. Let $X=\R\times S^1$ and consider the d-structure on $X$ given by the following condition: the path $\alpha=(\alpha_\R,\alpha_{S^1})\in P_X$ is a d-path if $\alpha_\R$ is non-decreasing, and $\alpha_\R(s)=\alpha_\R(t)$ implies $\alpha_{S^1}(s)=\alpha_{S^1}(t)$. We have
\begin{equation*}
	\vP_X((a,x),(b,y))\simeq
	\begin{cases}
		P_{S^1}(x,y)\simeq \Z & \text{for $a<b$,}\\
		\{*\} & \text{for $(a,x)=(b,y)$,}\\
		\emptyset & \text{for $a\geq b$, $(a,x)\neq(b,y)$.}
	\end{cases}
\end{equation*}	
	 Thus, $X$ has no loops, the pair $(X,X)$ is future stable (for any $\fF$) and $\vSP^+_\fF(X,X)$ is $\fF$--equivalent to a countable discrete space, which is not $\fF$--contractible.
\end{rem}

\begin{df}
	\emph{A stable total $\fF$--component system on $X$} is a family $\cA=\{A_i\}_{i\in I}$ that is both a stable future and a stable past component system, and such that every pair $(A_i,A_j)$ is totally $\fF$--stable.
\end{df}

\begin{prp}\label{p:CommonStabilizingPairs}
	If $\{A_i\}_{i\in I}$ is a finite total component system on $X$, then there exist $a_i,b_i\in A_i$, $i\in I$, such that for every $i,j\in I$, $(a_i,b_j)$ is a stabilizing pair of $(A_i,A_j)$.
\end{prp}
\begin{proof}
	Choose stabilizing pairs $(a_{ij},b_{ij})$ for every pair $(A_i,A_j)$. Since all components are future and past connected and there is a finite number of components, there exist points $a_i\in A_i$ and $b_j\in A_j$ such that $a_i\leq a_{ij}$ and $b_j\geq b_{ij}$ for all $i,j$. The conclusion follows from Proposition \ref{p:PropertiesOfStabilizingPairs}.(b).
\end{proof}

\begin{prp}\label{p:ComponentsSystemPartialOrder}
	Let $\{A_i\}_{i\in I}$ be a future $\fF$--component system on $X$. Then the relation on $I$
	\[
		i\leq j \; \Leftrightarrow \; \vSP^+_\fF(A_i,A_j)\neq \emptyset \; \buildrel{(\ref{p:TraceDominates})} \over\Leftrightarrow \; \vP_X(A_i,A_j)\neq \emptyset
	\]
	is a partial order.
\end{prp}
\begin{proof}
	The reflexivity of $\leq$ is obvious. If $i\leq j\leq k\in I$, then there exists a path $\alpha\in \vP(A_i,y)$, for some $y\in A_j$ and, by \ref{p:TraceDominates}, a path $\beta\in \vP(y,A_k)$. From the existence of $\alpha*\beta$ we conclude that $i\leq k$ which proves that the relation $\leq$ is transitive. If $i=k$, then also $j=i$, since $A_i=A_k$ is d-convex, which proves the antisymmetry.
\end{proof}

\section{The coarsest component systems}

Fix a d-space $X$ with no loops and an equivalence system $\fF$. We say that a stable future (resp. past, total) $\fF$--component system $\cA$ on $X$  is \emph{coarser} than $\cB$ if for every $B\in \cB$ there exists $A\in \cA$ such that $B\subseteq A$. The relation of being coarser is a partial order on the set of all future (resp. past, total) $\fF$--component systems on $X$.
It is not, in general, true that $X$ admits \emph{a coarsest future (past, total) $\fF$--component system}, i.e., the system that is coarser than any other system. However, if $X$ admits any finite future (past, total) $\fF$--component system, then a coarsest system on $X$ exists and this will be shown in this section.

Let $\cA=\{A_i\}_{i\in I}$ and $\cB=\{B_j\}_{j\in J}$ be finite stable future $\fF$--component systems on $X$.
Let $\sim$ be the equivalence relation on $I\coprod J$ spanned by
	\begin{equation}
		i\sim j \; \Leftrightarrow \; A_i\cap B_j \neq \emptyset,
	\end{equation}
	and let $K$ be the set of equivalence classes of the relation $\sim$.
	For $k\in K$ define	
	\begin{equation}
		C_k=\bigcup_{i\in k\cap I} A_i = \bigcup_{j\in k\cap J} B_j.
	\end{equation}
	The family $\cC=\cA\cup \cB=\{C_k\}_{k\in K}$ will be called \emph{the union} of the component systems $\cA$ and $\cB$. It is clear that the sets $C_k$ are pairwise disjoint and cover $X$.

\begin{thm}\label{t:SumOfComponentSystems}
	The family $\cC=\cA\cup \cB$ is a future $\fF$--component system on $X$. Furthermore, if $\cA$ or $\cB$ is closed, then $\cC$ is closed.
\end{thm}

The proof of Theorem \ref{t:SumOfComponentSystems} uses several lemmas; they will be formulated for the component system $\cA$ only, since similar facts hold for $\cB$. For $x\in X$, the index of the component of $\cA$ (resp. $\cB$) that contains $x$ will be denoted by $i(x)$ (resp. $j(x)$), so that $x\in A_{i(x)}$ and $x\in B_{j(x)}$. Recall that the sets $I$ and $J$ are partially ordered, as shown in Proposition \ref{p:ComponentsSystemPartialOrder}.

\begin{prp}\label{p:Ends}
	For every $i\in I$, there exists a unique element $j(i)\in J$ such that
	\[
		(A_i)_{\geq y_i}\subseteq B_{j(i)} \tag{*}
	\]
	for some $y_i\in A_i\cap B_{j(i)}$.
\end{prp}
\begin{proof}
	Assume otherwise. Fix $x_0\in A_i$. By the assumption, there exists $x_0<x_1\in A_i$ such that $j(x_0)<j(x_1)$. By repeating this argument, we construct an infinite sequence of points $x_0<x_1<x_2<\dots$ of $A_i$ such that $j(x_{r+1})\neq j(x_r)$. Since the sets $B_{j_r}$ are d-convex, all indices $j(x_r)$ must be different, which contradicts the finiteness of $J$. This proves the existence of $j(i)$ and $y_i$.
	
	If $j'(i)\in J$, $y_i'\in A_i$ is another pair satisfying the condition (*), then, since $A_i$ is future connected, there exists $y_i''\in (A_i)_{\geq y_i,y'_i}$. Then $y_i''\in A_i \cap B_{j(i)}\cap B_{j'(i)}$, which shows that $j(i)=j'(i)$. Thus, $j(i)$ is unique.
\end{proof}

\begin{prp}\label{p:MaximalEnd}
	Let $k\in K$, $i\in I\cap k$. The following conditions are equivalent:
	\begin{enumerate}[\normalfont (a)]
	\item{$i$ is a maximal element in $I\cap k$.}
	\item{$i$ is the unique maximal element in $I\cap k$.}
	\item{$A_i$ and $B_{j(i)}$ are cofinal.}
	\end{enumerate}
\end{prp}
\begin{proof}
	Assume that $A_i$ and $B_{j(i)}$ are not cofinal. Thus, by Proposition \ref{p:Ends} there exists a point $y_i\in A_i\cap B_{j(i)}$ such that $(A_i)_{\geq y_i}\subseteq B_{j(i)}$ and a point $z\in (B_{j(i)})_{\geq y_i} \setminus A_i$. Obviously $i<i(z)$ and $i(z)\in k$, since $y_i\in B_{j(i)}\cap A_{i}\neq\emptyset$ and $z\in B_{j(i)}\cap A_{i(z)}\neq\emptyset$. This contradicts the maximality of $i$ in $I\cap k$ and proves that (a) implies (c).
	
	Assume that $A_i$ and $B_{j(i)}$ are cofinal. Thus, there exists $y_i\in A_i\cap B_{j(i)}$ such that $U:=(A_i)_{\geq y_i}= (B_{j(i)})_{\geq y_i}$. Since $A_i$ is future connected, every path ending at some point of $A_i$ extends to a path ending at some point of $U$. Furthermore, by Proposition \ref{p:TraceDominates}, $\vP_X(x,A_i)\neq \emptyset$ if and only if $i(x)\leq i$. Therefore,
	\[
		(C_k)_{\leq U}=\{x\in C_k\;|\; \vP(x,U)\neq \emptyset\}=\{x\in C_k\;|\; \vP(x,A_i)\neq \emptyset\}=\bigcup_{\{g\in I\cap k\;|\; g\leq i\}} A_g.
	\]
	Similarly we obtain $(C_k)_{\leq U}=\bigcup_{\{h\in J\cap k\;|\; h\leq j(i)\}} B_h$, which implies that $(C_k)_{\leq U}=C_k$. As a consequence,
	\[
		I\cap k= \{g\in I\cap k\;|\; g\leq i\}
	\]
	and then $i$ is the unique maximal element of $I\cap k$.  Thus, (c) implies (b). The implication (b) $\Rightarrow$ (c) is obvious.
\end{proof}

The unique maximal element in $I\cap k$ (resp. $J\cap k$) will be denoted $i(k,+)$ (resp. $j(k,+)$). Note that $j(i(k,+))=j(k,+)$ and $i(j(k,+))=i(k,+)$.

\begin{prp}\label{p:CkCofinal}
	For every $k\in K$, $C_k$ is future connected and the sets $A_{i(k,+)}$, $B_{j(k,+)}$ and $C_k$ are cofinal.
\end{prp}
\begin{proof}
	If $x,y\in C_k$, then there exist $z_x\in (A_{i(k,+)})_{\geq x}$ and $z_y\in (A_{i(k,+)})_{\geq y}$, since $i(x),i(y)\leq i(k,+)$. Thus, there exists $z\in (A_{i(k,+)})_{\geq z_x,z_y}\subseteq (C_k)_{\geq x,y}$, since $A_{i(k,+)}$ is future connected.
	
	The cofinality of $A_{i(k,+)}$, $B_{j(k,+)}$ and $C_k$ is an immediate consequence of Proposition \ref{p:MaximalEnd}.
\end{proof}

\begin{prp}\label{p:CrossStabilization}
	For every $k\in K$ and every $j\in J$, the pair $(B_j,A_{i(k,+)})$ is future stable. Furthermore, if $l\in K$  then the spaces
	\begin{itemize}
	\item{
		$\vSP^+_\fF(B_j,A_{i(k,+)})$ for $j\in J\cap l$, and
	}
	\item{
		$\vSP^+_\fF(A_i,A_{i(k,+)})$ for $i\in I\cap l$
	}
	\end{itemize}
	are all $\fF$--equivalent.
\end{prp}
\begin{proof}
	Since $A_{i(k,+)}$ is cofinal with $B_{j(i(k,+))}=B_{j(k,+)}$, the first statement follows from \ref{p:CofinalityPreservesStability}. The second statement follows from \ref{p:CommonStableTraces}.
\end{proof}

\begin{prp}\label{p:CkIsConvex}
	For every $k\in K$, $C_k$ is d-convex.
\end{prp}
\begin{proof}
	Assume otherwise. Then there exist $k\neq l\in K$, $i_1,i_3\in k$, $i_2\in l$ and a d-path $\alpha\in \vP(A_{i_1},A_{i_3})$ such that $\alpha(\tfrac{1}{2})\in A_{i_2}$, i.e.,   $i_1<i_2<i_3$. Since $i_1<i_2\leq i(l,+)$ and $i_2<i_3\leq i(k,+)$, we have
	\begin{align*}
		\emptyset &\buildrel{(\ref{p:TraceDominates})}\over\neq \vSP^+_\fF(A_{i_1},A_{i(l,+)})\buildrel{(\ref{p:CrossStabilization})}\over\simeq \vSP^+_\fF(A_{i(k,+)},A_{i(l,+)})\\
		\emptyset &\buildrel{(\ref{p:TraceDominates})}\over\neq \vSP^+_\fF(A_{i_2},A_{i(k,+)})\buildrel{(\ref{p:CrossStabilization})}\over\simeq \vSP^+_\fF(A_{i(l,+)},A_{i(k,+)}).
	\end{align*}
	Thus $i(k,+)\leq i(l,+)\leq i(k,+)$, 	which contradicts Proposition \ref{p:ComponentsSystemPartialOrder}.
\end{proof}

\begin{prp}\label{p:PathToMaximal}
	For $k,l\in K$ and  $x,y\in C_k$ there exists a path $\alpha\in \vP_X(x,(C_k)_{\geq y})$ that future $\fF$--stabilizes in $C_l$.
\end{prp}
\begin{proof}
	Induction with respect to $i(x)\in I\cap k$. If $i(x)=i(k,+)$, then there exists $y'\in (A_{i(k,+)})_{\geq x,y}$, since $C_k$ is future connected and cofinal with $A_{i(k,+)}$. Any d-path $\alpha\in \vP_X(x,y')$ is contained in $A_{i(k,+)}$. Thus, $\alpha$ future $\fF$--stabilizes in $A_{i(l,+)}$ and, by cofinality, in $C_l$.
	
	Now assume that $i(x)$ is not maximal in $I\cap k$ and the proposition holds for all $i\in (I\cap k)_{> i(x)}$. Choose $y_{i(x)}\in A_{i(x)}\cap B_{j(i(x))}$ such that $(A_{i(x)})_{\geq y_{i(x)}}\subseteq B_{j({i(x)})}$ (\ref{p:Ends}). By Proposition \ref{p:MaximalEnd}, $A_{i(x)}$ and $B_{j({i(x)})}$ are not cofinal; thus, there exists an element $u\in (B_{j({i(x)})})_{\geq y_{i(x)}}$ that does not belong to $A_{i(x)}$. Choose $v\in (A_{i(x)})_{\geq y_{i(x)},x}$; clearly $v\in B_{j({i(x)})}$ and we can choose $x'\in (B_{j({i(x)})})_{\geq u,v}$. Notice that $i(x)<i(x')$, since $A_i$ is d-convex.

	Eventually, pick $\beta\in \vP_{A_{i(x)}}(x,v)$ and $\gamma\in \vP_{B_{j({i(x)})}}(v,x')$. From the inductive hypothesis, there exists a path $\alpha'\in \vP_X(x',(C_k)_{\geq y})$ which future $\fF$--stabilizes in $C_l$. Since $\beta$ future $\fF$--stabilizes in $A_{i(l,+)}$, $\gamma$ future $\fF$--stabilizes in $B_{j(l,+)}$ and all $A_{i(l,+)}$, $B_{j(l,+)}$ and $C_l$ are cofinal, both $\beta$ and $\gamma$ future $\fF$--stabilize in $C_l$. Proposition \ref{p:PropertiesOfStabilizers}.(\ref{i:PConcat}) implies that $\beta*\gamma*\alpha'$ future $\fF$--stabilizes in $C_l$.
\end{proof}

\begin{prp}\label{p:CkStable}
	For every $k,l\in K$, the pair $(C_k,C_l)$ is stable.
\end{prp}
\begin{proof}
	Fix $\alpha\in \vP_{C_k}$. Choose a point $z\in \St^+_\fF(\sigma_{\alpha(0)};A_{i(k,+)})$ and, applying Proposition \ref{p:PathToMaximal}, paths $\beta'\in \vP_X(\alpha(0),(C_k)_{\geq z})$, $\gamma'\in \vP_X(\alpha(1),(C_k)_{\geq z})$ that future $\fF$--stabilize in $C_l$. Next, choose $u\in (A_{i(k,+)})_{\geq \beta'(1),\gamma'(1)}$ and d-paths $\beta''\in \vP(\beta'(1),u)$, $\gamma''\in \vP(\gamma'(1),u)$. Since $\beta'',\gamma''\in \vP_{A_{i(k,+)}}$, they both future $\fF$--stabilize in $A_{i(k,+)}$ and hence in $C_l$ (\ref{p:CkCofinal}, \ref{p:CofinalityPreservesStability}); thus, also $\beta=\beta'*\beta''$ and $\gamma=\gamma'*\gamma''$ future $\fF$--stabilize in $C_l$ (\ref{p:PropertiesOfStabilizers}.(\ref{i:PConcat})). Choose $v\in \St^+_\fF(\beta,\gamma;C_l)$.

Since $u\geq z$, we have $u\in \St^+_\fF(\alpha(0);A_{i(k,+)})$ and then
\[
	\vP(\alpha(0),u)\buildrel\fF\over\simeq \vSP^+_\fF(A_{i},A_{i(k,+)})\buildrel\fF\over\simeq \vSP^+_\fF(A_{i(k,+)},A_{i(k,+)})\buildrel\fF\over\simeq\{*\},
\]
where $i\in I\cap k$ is the index such that $\alpha(0)\in A_i$. 
As a consequence, $\vP(\alpha(0),u)$ is path-connected and then $\alpha*\gamma\sim\beta$. Thus, for every d-path $\omega\in \vP_{(C_l)_\geq v}$,
 the diagram
\[
	\begin{diagram}
		\node{\vP(u,\omega(0))}
			\arrow{e,tb}{\vP(\beta,\omega)}{\simeq}
			\arrow{s,lr}{\vP(\gamma,\sigma_{\omega(0)})}{\simeq}
		\node{\vP(\alpha(0),\omega(1))}
	\\
		\node{\vP(\alpha(1),\omega(0))}
			\arrow{ne,r}{\vP(\alpha,\omega)}
	\end{diagram}
\]
commutes up to homotopy. Since $\omega(0)\geq v$ and $v\in \St^+_\fF(\beta,\gamma;C_l)$, the maps marked with $\simeq$ are $\fF$--equivalences. Thus, $\vP(\alpha,\omega)$ is also an $\fF$--equivalence and then $v\in \St^+_\fF(\alpha;C_l)\neq \emptyset$.
\end{proof}

\begin{proof}[Proof of Theorem \ref{t:SumOfComponentSystems}]
	For every $k\in K$, $C_k$ is future connected (\ref{p:CkCofinal}), d-convex (\ref{p:CkIsConvex}) and every pair $(C_k,C_l)$, $l\in K$, is future $\fF$--stable (\ref{p:CkStable}). Proposition \ref{p:CommonStableTraces} implies that
	\[
		\vSP^+(C_k,C_k)\buildrel\fF\over\simeq \vSP^+(A_{i(k,+)},A_{i(k,+)})\buildrel\fF\over\simeq \{*\},
	\]
	so $C_k$ is future $\fF$--trivial.
	
	If $\cA$ is closed, then the final point of $A_{i(k,+)}$ is a final in $C_k$. Hence, $\cC$ is closed.
\end{proof}

\begin{prp}\label{p:SumOfTotalSystem}
	If $\cA$ and $\cB$ are stable total $\fF$--component systems on $X$, then  $\cC=\cA\cup \cB$ is also a stable total $\fF$--component system.
\end{prp}
\begin{proof}
	Theorem \ref{t:SumOfComponentSystems} implies that $\cC$ is both a future and a past $\fF$--component system. For every $k,l\in K$, there are points $x\in A_{i(k,-)}$ and $y\in A_{i(l,+)}$ such that $y\in \St^+_\fF(x;A_{i(l,+)})$ and $x\in \St_\fF^-(A_{i(k,-)};y)$. The index $i(k,-)$ denotes the unique minimal element of $I\cap k$.  Since $A_{i(l,+)}$ and $C_l$ are cofinal and $A_{i(k,-)}$ and $C_k$ are coinitial, there exist $x'\in (A_{i(k,-)})_{\leq x}$ and $y'\in (A_{i(l,+)})_{\geq y}$ such that $(A_{i(k,-)})_{\leq x'}=(C_k)_{\leq x'}$ and $(A_{i(l,+)})_{\geq y'}=(C_l)_{\geq y'}$. Now $x'\in \St^-_{\fF}(C_k;y')$ and $y'\in \St^+_{\fF}(x';C_l)$; therefore $(C_k,C_l)$ is totally stable.
\end{proof}

\begin{thm}\label{t:UniqueSystem}
	If $X$ admits a finite future (resp., past, total) $\fF$--component system, then there exists a coarsest $\fF$--component system on $X$. If $X$ admits a finite future (resp. past) closed $\fF$--component system, then the coarsest future (resp. past) $\fF$--component system is closed.
\end{thm}
\begin{proof}
	The existence of a finite future $\fF$--component system implies the existence of a minimal system $\cA$, i.e., such that no future $\fF$--component system is coarser than $\cA$.
	If $\cA$ and $\cB$ are minimal future $\fF$--components systems, then, by Theorem \ref{t:SumOfComponentSystems},  the sum $\cA\cup \cB$ is a future component system that is coarser than both $\cA$ and $\cB$. Thus, $\cA\cup\cB=\cA=\cB$ is a unique minimal system.
	The same argument applies for past component systems and, by Proposition \ref{p:SumOfTotalSystem}, for total systems.
	
	If $\cA$ is the coarsest future $\fF$--component system on $X$, and $\cB$ is a closed one, then $\cA=\cA\cup \cB$ is closed.
\end{proof}

\begin{rem}\label{r:Counterexample}
Theorem \ref{t:SumOfComponentSystems} is not valid for infinite component systems.	The d-space
 \begin{equation}
	X=\vI \times \vR/\{(1,y)\sim (0,y+1)\}_{y\in \R}
\end{equation}
admits total component systems
\begin{equation}
	\cA=\{[0,1)\times [n,n+1)\}_{n\in\Z} \qquad \text{and} \qquad \cB=\{[0,1)\times [n-\tfrac{1}{2},n+\tfrac{1}{2})\}_{n\in\Z}
\end{equation}
but $\cA\cup \cB=\{X\}$ is neither a future nor a past component system, since $(X,X)$ is neither future nor past stable. In particular, $X$ admits no coarsest component system.
\end{rem}

The coarsest stable future (resp. past, total) $\fF$--component system on $X$ will be denoted by $\cM^+_\fF(X)$ (resp. $\cM^-_\fF(X)$, $\cM^\pm_\fF(X)$).

\section{Component systems on pre-cubical sets}
\label{s:Cubical}

In this section, we prove that every finite pre-cubical set $K$ with no loops admits a finite total (resp. past, total) $\fF$--component system and hence, by Theorem \ref{t:UniqueSystem}, a coarsest one. Moreover, we show that the coarsest finite future (resp. past) $\fF$--component system on $|K|$ is closed. Since every $\fF_\infty$--component system is an $\fF$--component system for every equivalence system $\fF$, we assume $\fF=\fF_\infty$.

Here we recall basic definitions; for a survey on pre-cubical sets and their applications in concurrency, see eg. \cite{FGHMR,FGR}.
\begin{df}
	\emph{A pre-cubical set} $K$ is a sequence of disjoint sets $(K[n])_{n\geq 0}$ equipped with \emph{face maps} $d^\varepsilon_i:K[n]\to K[n-1]$, $1\leq i\leq n$, $\varepsilon\in\{0,1\}$, satisfying the \emph{pre-cubical relations} $d^\varepsilon_i d^\eta_j=d^\eta_{j-1}d^\varepsilon_i$ for $i<j$. The elements of $K[n]$ will be called \emph{$n$--cubes}.  \emph{The geometric realization} $|K|$ of a pre-cubical set $K$ is the quotient d-space
	\begin{equation}
		|K|=\coprod_{n\geq 0} K[n]\times \vI^n/\sim
	\end{equation}	
	where $\sim$ is generated by
	\begin{equation}
		 [d^\varepsilon_i(c),(t_1,\dots,t_{n-1})]\sim [c,(t_1,\dots,t_{i-1},\varepsilon,t_i,\dots,t_{n-1})]
	\end{equation}	
	for $n\geq 0$, $c\in K[n]$, $1\leq i\leq n$ and $\varepsilon\in\{0,1\}$.
\end{df}

Let $K$ be a pre-cubical set. For an $n$--cube $c\in K[n]$, let $d^0(c)=(d^0_1)^n(c)\in K[0]$ and $d^1(c)=(d^1_1)^n(c)\in K[0]$ denote the initial and the final vertex of $c$, respectively.

\begin{exa}\label{x:nCube}
\emph{The standard $n$--cube $\square^n$} is the pre-cubical set whose $k$--cubes are defined by
\begin{equation}
	\square^n[k]=\{(a_1,\dots,a_n)\in \{0,*,1\}^n\; | \; \card(\{i: a_{i}=*\})=k\},
\end{equation}  
and the face map $d^\varepsilon_i$ converts the $i$--th occurrence of $*$ into $\varepsilon$. The geometric realization of $\square^n$ is d-homeomorphic to the directed $n$--cube $\vI^n$.
\end{exa}

We say that a pre-cubical set \emph{has no loops} if the following conditions are satisfied, which are equivalent:
\begin{itemize}
	\item{$|K|$ has no loops,}
	\item{there is no non-empty sequence of $1$--cubes $e_1,\dots,e_n$ such that $d^1(e_n)=d^0(e_1)$ and $d^1(e_i)=d^0(e_{i+1})$ for $1\leq i<n$.}
\end{itemize}

Every point $x\in |K|$ has a unique presentation $x=[c;(t_1,\dots,t_n)]$ such that $c\in K[n]$ and $0< t_i < 1$ for all $i$. For every $n$--cube $c\in K[n]$, define the set
\begin{equation}\label{e:Uc}
        U_c=\{[c;(t_0,\dots,t_n)]\;|\; \forall_{i}\; 0<t_i<1 \} \subseteq |K|
\end{equation}
For a vertex ($0$--cube) $v\in K[0]$, let
\begin{equation}
       V^+_v=\bigcup_{c\in K\;|\; d^1(c)=v} U_c,\qquad V^-_v=\bigcup_{c\in K\;|\; d^0(c)=v} U_c.
\end{equation}
The point $v=[v;()]$ is a final point of $V^+_v$ and an initial point of $V^-_v$.

\begin{exa}\label{x:UVCube}
	For $K=\square^n$,
	\begin{itemize}
		\item{$U_{(a_1,\dots,a_n)}=J_{a_1}\times \dots\times J_{a_n}\subseteq \vI^n$, where $J_0=\{0\}$, $J_1=\{1\}$ and $J_*=(0,1)$.}
		\item{$V^+_{(a_1,\dots,a_n)}=J^+_{a_1}\times \dots\times J^+_{a_n}\subseteq \vI^n$, where $J^+_0=\{0\}$, $J^+_1=(0,1]$.}
		\item{$V^-_{(a_1,\dots,a_n)}=J^-_{a_1}\times \dots\times J^-_{a_n}\subseteq \vI^n$, where $J^-_0=[0,1)$, $J^-_1=\{1\}$.}
	\end{itemize}
\end{exa}

The main result is the following theorem:

\begin{thm}\label{t:InteriorSystem}
	Let $K$ be a pre-cubical set with no loops. For every equivalence system $\fF$,
	\begin{enumerate}[\normalfont (a)]
	\item{$\{U_c\}_{c\in \bigcup K[n]}$ is a total $\fF$--component system on $|K|$,}
	\item{$\{V^+_v\}_{v\in K[0]}$ is a future $\fF$--component system on $|K|$,}
	\item{$\{V^-_v\}_{v\in K[0]}$ is a past $\fF$--component system on $|K|$.}
	\end{enumerate}		
\end{thm}

The proof of Theorem \ref{t:InteriorSystem} uses the following:

\begin{prp}\label{p:CubPathHEq}
	Let $K$ be a pre-cubical set. Let $\gamma=(\gamma_1,\dots,\gamma_n)\in \vP_{\vI^n}$ denote a path such that $\gamma_i(0)>0$ for all $1\leq i\leq n$ and, for some $0\leq r\leq 1$, $\gamma_i(1)\leq r$ for all $i$. Fix $a\in  K[n]$ and denote $\alpha=[a;\gamma]\in \vP_{|K|}$. Then, for every $z=[b;(t_1,\dots,t_m)]$, $b\in K[m]$, such that $r\leq t_i$ for all $i$, the map
	\[
		F=\vP(\alpha,z): \vP_{|K|}(\alpha(1),z)\ni \omega \mapsto \alpha * \omega \in \vP_{|K|}(\alpha(0),z)
	\]
	is a homotopy equivalence.
\end{prp}
\begin{proof}
	First we will consider the case when $\gamma_i(1)=r$ for all $i$.
	Denote $q=\min_{1\leq i\leq n} \gamma_i(0)$. Let $h:[0,1]\to[0,1]$ be a non-decreasing continuous function satisfying  $h(q)=r$ and $h(x)=x$ for $x\in \{0\}\cup [r,1]$ and let $h_s(x)=(1-s)x+sh(x)$. The function $h_s$ induces a self-deformation of $|K|$ via continuous d-maps, given by the formula
	\[
        h_s^K: |K|\ni [c;(t_1,\dots,t_k)] \mapsto [c; h_s(t_1),\dots, h_s(t_k))] \in |K|
    \]
    for all $k\geq 0$ and $c\in K[k]$. Clearly $h_1^K(\alpha(0))=\alpha(1)=[a; (r,\dots,r)]$ and $h_1^K(z)=z$, so we can define the continuous map
	\[
		G:\vP_{|K|}(\alpha(0),z) \ni  \omega \mapsto h_1^K\circ \omega \in \vP_{|K|}(\alpha(1),z).
	\]
	We will show that $F$ and $G$ are homotopy inverses. We have
	\[
		(G\circ F)(\omega)=G(\alpha*\omega)=h^K_1\circ (\alpha*\omega)=(h^K_1\circ \alpha)*(h^K_1\circ \omega)=\sigma_{\alpha(1)}*(h^K_1\circ\omega).
	\]
	Thus, $G\circ F$ is homotopic to the map $\omega\mapsto h^K_1\circ \omega$ and then, by the self-deformation $h^K_s$, to the identity on $\vP(\alpha(1),z)$.
	Next,
	\[
		(F\circ G)(\omega)=F(h^K_1\circ \omega) = \alpha * (h^K_1\circ \omega)
	\]
	The formula
	\[
		H(\omega,s) = (c;(1-s)\alpha+s\sigma_{\alpha(0)})*(h^K_{1-s}\circ \omega)
	\]
	defines the homotopy between $F\circ G$ and $\omega \mapsto \sigma_{\alpha(0)}*\omega$, and the latter map is homotopic to the identity on $\vP(\alpha(0),z)$.
	
	To prove the general case, choose any paths $\delta\in\vP_{\vI^n}(\gamma(0),(r,\dots,r))$, $\delta'\in\vP_{\vI^n}(\gamma(1),(r,\dots,r))$. Since $\delta\sim \gamma*\delta'$, the diagram
\begin{equation}
	\begin{diagram}
		\node{}
		\node{\vP_{|K|}([a;(r,\dots,r)],z)}
			\arrow{se,tb}{\vP([a;\delta],z)}{\simeq}
			\arrow{sw,tb}{\vP([a;\delta'],z)}{\simeq}
		\node{}
	\\
		\node{\vP_{|K|}([a;\gamma(1)],z)}
			\arrow[2]{e,t}{\vP([a;\gamma],z)}
		\node{}
		\node{\vP_{|K|}([a;\gamma(0)],z)}
	\end{diagram}
\end{equation}
commutes up to homotopy. It is already shown that the upper maps are homotopy equivalences.
Thus, $\vP_{|K|}([c;\gamma],z)=\vP_{|K|}(\alpha,z)$ is also a homotopy equivalence.
\end{proof}

A slight generalization of Proposition \ref{p:CubPathHEq} reads as follows:
\begin{prp}\label{p:CubPaths}
	Let $K$ be a pre-cubical set with no loops. Choose  $a\in K[n]$, $b\in K[m]$, $\gamma\in \vP_{\vI^n}$, $\delta\in \vP_{\vI^m}$ and let $\alpha=(a;\gamma)$, $\beta=(b;\delta)$. Assume that there exists $0<r<1$ such that $0<\gamma_i(1)\leq r\leq \delta_{j}(0)<1$ for all $1\leq i\leq n$ and $1\leq j\leq m$. Then the map
	 \[
	 	\vP_{|K|}(\alpha,\beta):\vP_{|K|}(\alpha(1),\beta(0)) \to \vP_{|K|}(\alpha(0),\beta(1)) \tag{*}
	 \]
	 is a homotopy equivalence.
\end{prp}
\begin{proof}
	The map (*) is the composition
	\[
		\vP_{|K|}(\alpha(1),\beta(0)) \xrightarrow{\vP(\alpha,\beta(0))} \vP_{|K|}(\alpha(0),\beta(0)) \xrightarrow{\vP(\alpha(0),\beta)} \vP_{|K|}(\alpha(0),\beta(1)).
	\]
	Proposition \ref{p:CubPathHEq} applied to $\alpha$ and $z=\beta(0)$ implies that $\vP(\alpha,\beta(0))$ is a homotopy equivalence. Let $K^{op}$ be the pre-cubical complex opposite to $K$, i.e., $K^{op}[n]=K[n]$, $(d^\varepsilon_i)^{op}=d^{1-\varepsilon}_i$. There is an obvious d-homeomorphism
	\[
		|K|^{op}\ni (b;t_1,\dots,t_k) \mapsto (b;1-t_1,\dots,1-t_k)\in |K^{op}|.
	\]
	Now \ref{p:CubPathHEq} applied to $\beta^{op}$, where $\beta^{op}(t)=\beta(1-t)$, and $z=\alpha(0)$ on $K^{op}$ implies that $\vP(\alpha(0),\beta)$ is a homotopy equivalence.
\end{proof}

\begin{proof}[Proof of \ref{t:InteriorSystem}]
	(a):
	It is easy to check that the sets $U_c$ are pair-wise disjoint, future and past connected, future and past $\fF$--trivial and that they cover $|K|$. It remains to prove that every pair $(U_a,U_b)$ is future, past and totally $\fF$--stable for all $a\in K[n]$, $b\in K[m]$.

	Fix a path $\alpha\in \vP_{U_c}$; there is a presentation $\alpha=[a;\gamma]$ with $\gamma=(\gamma_i)_{i=1}^n\in \vP_{\vI^n}$. Clearly $0<\gamma_i(0)\leq \gamma_i(1)<1$  for all $i$, and then $r=\max_{i=1}^n \gamma_i(1)<1$. Let $y=[b;(r,\dots,r)]$; for every path $\beta\in\vP_{(U_{b})_{\geq y}}$, Proposition \ref{p:CubPaths} implies that $\vP_{|K|}(\alpha,\beta)$ is an $\fF$--equivalence. Thus, $y\in \St^+_\fF(\alpha;U_{b})$, which implies that the pair $(U_a,U_{b})$ is future $\fF$--stable. A similar argument shows that this pair is past stable.
	
	To prove total stability, note that $([a;(\tfrac{1}{2},\dots,\tfrac{1}{2})],[b;(\tfrac{1}{2},\dots,\tfrac{1}{2})])$ is a total stabilizer of $(U_a,U_b)$, again by Proposition \ref{p:CubPaths}.

	(b) and (c):
	Again, the only non-trivial part to check is the stability of all pairs $(V^+_v,V^+_w)$, $v,w\in K[0]$. Since $w$ is the final point of $V^+_w$, it remains to check that $\vP(\alpha,w)$ is an $\fF$--equivalence for $\alpha\in \vP_{V_v^+}$. Since $\alpha$ lies in a single cube, i.e., $\alpha=[c;\gamma]$ for $c\in K[n]$ and $\gamma\in \vP_{\vI^n}$, then this follows from Proposition \ref{p:CubPathHEq}.
\end{proof}

Combining Theorems \ref{t:UniqueSystem} and \ref{t:InteriorSystem} we obtain
\begin{cor}
	Every finite pre-cubical set $K$ with no loops admits, for every equivalence system $\fF$:
	\begin{enumerate}[\normalfont (a)]
	\item{A coarsest total stable $\fF$--component system $\cM^\pm_{\fF}(K)$.}
	\item{A coarsest future (resp. past) stable $\fF$--component system $\cM^+_{\fF}(K)$ (resp. $\cM^-_{\fF}(K)$), which is closed.}
	\end{enumerate}
\end{cor}

The component systems given in Theorem \ref{t:InteriorSystem} may or may not be the coarsest; some examples are presented in the last section.

\section{Stable component category}
\label{s:ComponentCategory}

In this section we construct the component category associated to a stable (future/past/total) $\fF$--component system. The objects of the component category are the components of the component system, and the morphisms carry information about the stable path space between particular components. Component categories are enriched in some category $\cS$, i.e., their "morphism sets" are objects of $\cS$. In most typical cases, $\cS$ will be the homotopy category of topological spaces, the category of graded $R$--modules for a principal ideal domain $R$  or the category of sets (in the last case, we obtain categories in the usual sense).

We refer the reader to the book by Kelly \cite{Kelly} for the basic definitions concerning enriched categories.

Let us fix an equivalence system $\fF$, a monoidal category $\cS$ and a monoidal functor $L:\Sp\to \cS$ that sends $\fF$--equivalences to isomorphisms in $\cS$. Equivalently, we require that there is a factorization
\begin{equation}
	\begin{diagram}
		\node{}
		\node{\Sp[\fF^{-1}]}
			\arrow{se,t}{\tilde{L}}
		\node{}
	\\
		\node{\Sp}
			\arrow[2]{e,t}{L}
			\arrow{ne}
		\node{}
		\node{\cS}
	\end{diagram}
\end{equation}

Fix a d-space $X$ with no loops and a stable future $\fF$--component system $\cA:=\{A_i\}_{i\in I}$ on $X$. As noted before, the indexing set $I$ is partially ordered by $i\leq j$ if and only if $\vP_X(A_i,A_j)\neq \emptyset$.

In the case when $\fF=\fF_\infty$ is the family of weak homotopy equivalences and all components $A_i$ have the maximal elements $a_i\in A_i$ we can define the future component category as the $\Sp$--category with $I$ as the set of objects and $\vT_X(a_i,a_j)$ as the morphism object from $i$ to $j$. This is a full subcategory of the trace category of $\vT(X)$ defined by Raussen \cite[Section 3]{RInv}. However, we do not want to make such an assumption. In most cases, one can restrict to closed future component systems but we want to obtain a definition of a component category that works also for total component systems. Apart from the most trivial cases, d-spaces do not admit total component systems having components with final points.

\begin{df}\label{d:FutureSystemOfRepresentatives}
	\emph{A future system of representatives} of $\cA$ is a collection of points $\{a_i\in A_i\}_{i\in I}$ such that $a_j\in \St^+_\fF(a_i,A_j)$ for every $i<j\in I$.
\end{df}

\begin{rem}
    Obviously, $a_j\in \St^+_\fF(a_i,A_j)$ if $i\not\leq j$, since $\vP(A_i,A_j)=\emptyset$ in this case. However, we do not assume that $a_i\in \St^+(a_i,A_i)$.
\end{rem}

Note that if $\{a_i\}_{i\in I}$ is a future system of representatives of $\cA$, then $\vSP_\fF^+(A_i,A_j)$ and $\vP(a_i,a_j)$ are $\fF$--equivalent for all $i,j\in I$. Indeed, for $i< j$ this follows immediately from the definition, for $i=j$ both $\vSP_\fF^+(A_i,A_i)$ and $\vP(a_i,a_i)\cong\{*\}$ are $\fF$--contractible, and for $i\not\leq j$ both $\vSP_\fF^+(A_i,A_i)$ and $\vP(a_i,a_i)$ are empty.

\begin{prp}\label{p:ExistenceOfRepresentants}
	Assume that the component system $\cA$ is finite. For any collection of points $\{b_i\in A_i\}_{i\in I}$, there exists a future system of representatives $\{a_i\}_{i\in I}$ of $\cA$ such that $a_i\geq b_i$ for all $i\in I$.
\end{prp}
\begin{proof}
	The points $a_i$ can be chosen inductively. If $j\in I$ is minimal, let $a_j\in (A_j)_{\geq b_j}$ be an arbitrary point. If $j$ is not minimal and all points $a_i\in A_i$ for $i<j$ are chosen according to Definition \ref{d:FutureSystemOfRepresentatives}, then let $a_j$ be an arbitrary point of the set
 \[
    \bigcap_{i\in I_{<j}} \St^+(a_i;A_j)_{\geq b_j},
 \]
    which is non-empty by Proposition \ref{p:PropertiesOfStabilizers}.
\end{proof}

\begin{rem}
    There exist infinite future component systems that does not admit any system of representatives. For example, consider the d-space
    \[
    	X=\{(x,y)\in \vR^2\;|\; x+y\geq 0\}
    \]
    with the component system $\{A_{i}\}_{i\in \Z}$, $A_i=\{(x,y)\in X\;|\; i\leq x < i+1\}$. For any choice of points $a_i=(a_i^x,a_i^y)\in A_i$ we have $\lim_{i\to -\infty}a_i^y=+\infty$ and, therefore, the space $\vP(a_i,a_j)$ is empty for every $j\in \Z$ if $i<j$ is small enough. On the other hand, $\vSP^+(A_j,A_i)\simeq \{*\}$ is non-empty.
\end{rem}

\begin{df}
	Let $\ba=\{a_i\}_{i\in I}$ be a future system of representatives of $\cA$. \emph{The future component category} of $\cA$ with representatives $\ba$ is the $\cS$--category $\cT^{\ba}=\cT_{\cA,L}^{\ba}$ with
	\begin{itemize}
	\item{$I$ as the set of objects,
	}
	\item{$\cT^{\ba}(i,j)=L(\vP(a_i,a_j))$ as the morphism objects,
	}
	\item{$L(\{*\}\to \{\sigma_{a_i}\}=\vP(a_i,a_i))$ as the identities,
	}
	\item{For $i<j<k\in I$, the composition is given by
\[
	L(\vP(a_i,a_j))\otimes L(\vP(a_j,a_k)) \to L(\vP(a_i,a_j)\times \vP(a_j,a_k)) \xrightarrow{L(c_{ijk})} L(\vP(a_i,a_k)),
\]	
	where $c_{ijk}$ stands for the concatenation map. The left-hand morphism is the structure morphism of the monoidal functor $L$.
	}
	\end{itemize}
\end{df}

This definition make sense; the only non-trivial thing to check is that the composition is associative and this follows immediately from properties of monoidal functors  and the fact that the concatenation of d-paths is associative up to homotopy.

Equivalently, we can define $\cT^\ba_\cA$ as $L_*(\vT(X)|_{\ba})$, where $\vT(X)|_{\ba}$ is the full subcategory of the trace category with the objects $\{a_i\}_{i\in I}$, and $L_*:\Sp\mhyphen\Cat\to \cS\mhyphen\Cat$ is the functor changing the enriching category induced by $L$.

Our next goal is to prove that the $\cS$--category $\cT^\ba$ does not depend on the choice of the system of representatives $\ba$ of $\cA$.

\begin{prp}\label{p:NaturalEquiv}
	Let $\ba=\{a_i\}_{i\in I}$ and $\bb=\{b_i\}_{i \in I}$ be collections of points in $X$ and let $\{\omega_i\in \vP(a_i,b_i)\}_{i\in I}$ be a collection of d-paths. Assume that both maps
	\begin{align*}
		\vP(\omega_i,b_j)&:  \vP(b_i,b_j)\to \vP(a_i,b_j)\\
		\vP(a_i,\omega_j)&:  \vP(a_i,a_j)\to \vP(a_i,b_j)
	\end{align*}	
	are $\fF$--equivalences for every $i,j\in I$. Then the formula
	\begin{equation}\label{e:Fij}
		F_{i,j}: L(\vP(a_i,a_j))\xrightarrow{L(\vP(a_i,\omega_j))} L(\vP(a_i,b_j))\xrightarrow{L(\vP(\omega_i,b_j))^{-1}} L(\vP(b_i,b_j)),
	\end{equation}
	defines an equivalence of the $\cS$--categories $L_*(\vT(X)|_{\ba})$ and $L_*(\vT(X)|_{\bb})$.
\end{prp}
\begin{proof}
	Clearly, for every $i,j\in I$, the morphism $F_{i,j}$ is an isomorphism in $\cS$ and hence we can define $G_{i,j}=F^{-1}_{i,j}$. 
	It remains to prove that the collections of morphisms $\{F_{i,j}\}$ and $\{G_{i,j}\}$ define functors $F$ and $G$, respectively; if so, $F$ and $G$ are mutual inverses. This will be shown only for $F$; the argument for $G$ is similar. It is clear that identities are preserved by $F$. To show that $F$ preserves compositions of morphisms, we need to check that the diagram below is commutative, for all $i\leq j\leq k \in I$:
\begin{equation}\label{e:FunctorComposition}
	\begin{diagram}
		\node{L(\vP(a_i,a_j))\otimes L(\vP(a_j,a_k))}
			\arrow{e}
			\arrow{s,r}{\simeq}
		\node{L(\vP(a_i,a_j)\times\vP(a_j,a_k))}
			\arrow{e}
			\arrow{s,r}{\simeq}
		\node{L(\vP(a_i,a_k))}
			\arrow{s,r}{\simeq}
	\\
		\node{L(\vP(a_i,b_j))\otimes L(\vP(a_j,b_k))}
			\arrow{e}
			\arrow{s,r}{\simeq}
		\node{L(\vP(a_i,b_j)\times\vP(a_j,b_k))}
			\arrow{s,r}{\simeq}
		\node{L(\vP(a_i,b_k))}
			\arrow{s,r}{\simeq}
	\\
		\node{L(\vP(b_i,b_j))\otimes L(\vP(b_j,b_k))}
			\arrow{e}
		\node{L(\vP(b_i,b_j)\times\vP(b_j,b_k))}
			\arrow{e}
		\node{L(\vP(b_i,b_k))}
	\end{diagram}
\end{equation}	
All vertical morphisms are induced by concatenation with one of the paths $\omega_i, \omega_j,\omega_k$, or are inverses of such morphisms. They are all isomorphisms, since the products of $\fF$--equivalences are $\fF$--equivalences and $L$ sends $\fF$--equivalences into isomorphisms. The left-hand part of this diagram commutes since $L$ is a monoidal functor.	To prove that the right-hand part commutes, consider the diagram of spaces:
	\begin{equation}	
	\begin{diagram}
		\node{\vP(a_i,a_j)\times \vP(a_j,a_k)}
			\arrow[2]{e,t}{f}
			\arrow{se,tb}{k}{\simeq}
			\arrow[2]{s,lr}{a}{\simeq}
		\node{}
		\node{\vP(a_i,a_k)}
			\arrow[2]{s,lr}{\simeq}{c}
	\\
		\node{}
		\node{\vP(a_i,a_j)\times\vP(a_j,b_k)}
			\arrow{sw,tb}{l}{\simeq}
			\arrow{se,t}{h}
		\node{}
	\\
		\node{\vP(a_i,b_j)\times \vP(a_j,b_k)}
		\node{\vP(a_i,a_j)\times \vP(b_j,b_k)}
			\arrow{w,tb}{r}{\simeq}
			\arrow{n,lr}{q}{\simeq}
			\arrow{s,lr}{s}{\simeq}
		\node{\vP(a_i,b_k)}
	\\
		\node{}
		\node{\vP(a_i,b_j)\times\vP(b_j,b_k)}
			\arrow{nw,tb}{m}{\simeq}
			\arrow{ne,t}{j}
		\node{}
	\\
		\node{\vP(b_i,b_j)\times\vP(b_j,b_k)}	
			\arrow[2]{e,t}{g}
			\arrow[2]{n,lr}{b}{\simeq}
			\arrow{ne,tb}{n}{\simeq}
		\node{}
		\node{\vP(b_i,b_k)}
			\arrow[2]{n,lr}{\simeq}{d}
	\end{diagram}
	\end{equation}
	in which, again, every map is given either by concatenation or by pre- or post-composition with one of the paths $\omega_i$, $\omega_j$, $\omega_k$. All maps marked $\simeq$ are $\fF$--equivalences for the reasons mentioned before. It is elementary to check that all triangles and squares in this diagram commute up to homotopy and then, after composing with $L$, they commute strictly. Thus,
	\begin{multline*}
		L(d)^{-1}L(c)L(f)=
		L(d)^{-1}L(h)L(k)=
		L(d)^{-1}L(h)L(l)^{-1}L(a)=\\
		L(d)^{-1}L(h)L(q)L(r)^{-1}L(a)=
		L(d)^{-1}L(j)L(s)L(r)^{-1}L(a)=
		L(d)^{-1}L(j)L(m)^{-1}L(a)=\\
		L(d)^{-1}L(j)L(n)L(b)^{-1}L(a)=
		L(d)^{-1}L(d)L(g)L(b)^{-1}L(a)=
		L(g)L(b)^{-1}L(a),
	\end{multline*}
	which shows that also the right-hand part of (\ref{e:FunctorComposition}) commutes.
\end{proof}

\begin{prp}\label{p:TraceCatIndependent}
	The category $\cT^\ba$ does not depend, up to isomorphism,  on the choice of a future representative system $\ba$ of $\cA$.
\end{prp}
\begin{proof}
	Given two future representative systems $\ba=\{a_i\}$ and $\ba'=\{a'_i\}$ of $\cA$, one can find a representative system $\bb=\{b_i\}$ and paths $\omega_i\in\vP(a_i,b_i)$ and $\omega'_i\in \vP(a'_i,b_i)$ such that Proposition \ref{p:NaturalEquiv} applies both for $\ba$, $\bb$ and $\{\omega_i\}$ and for $\ba'$, $\bb$ and $\{\omega'_i\}$. This can be done inductively in a similar way as in the proof of \ref{p:ExistenceOfRepresentants}: if $j\in I$ is minimal, we choose any $b_j\in \St^+_\fF(a_j,a_j';A_j)$ and any $\omega_j\in \vP(a_j,b_j)$, $\omega'_j\in\vP(a_j',b_j)$. If $j$ is not minimal and $b_i$, $\omega_i$, $\omega'_i$ are chosen for all $i<j$ according to Proposition \ref{p:NaturalEquiv}, let $b_j$ be an element of
	\begin{equation}
		\St^+_\fF(a_j,a'_j;A_j)\cap \bigcap_{i<j}\St^+_\fF(\omega_i;A_j),
	\end{equation}
	which is non-empty by \ref{p:PropertiesOfStabilizers}. Now, Proposition \ref{p:NaturalEquiv} provides natural equivalences of $\cS$--categories
	\[
		\cT^\ba \xrightarrow\simeq \cT^\bb \xleftarrow\simeq \cT^{\ba'}.\qedhere
	\]
\end{proof}

Propositions \ref{p:ExistenceOfRepresentants} and \ref{p:TraceCatIndependent} allow to formulate the following.

\begin{df}\label{d:ComponentCategory}
    \emph{The future $\fF$--component category} of a finite stable future $\fF$--component system is the $\cS$--category
    \begin{equation*}
	   \vSP^+_{\fF,L}(X,\cA):= \cT^\ba
    \end{equation*}
    for any future system of representatives $\ba$ of $\cA$. Similarly, for a past $\fF$--component system $\cB$ we define its \emph{past $\fF$--component category} as $\vSP^-_{\fF,L}(X,\cB)=\cT^{\bb}$ for a past system of representatives $\bb$.
\end{df}

Given a finite total $\fF$--component system $\cA=\{A_i\}_{i\in I}$, one can define the future (resp. past) component category $\vSP^+_{\fF,L}(X,\cA)$ (resp. $\vSP^-_{\fF,L}(X,\cA)$). A natural requirement is that these $\cS$--categories should be equivalent.

By Proposition \ref{p:CommonStabilizingPairs}, there exist $c_i,d_i\in A_i$, $i\in I$ such that $d_j\in \St^+_\fF(c_i;A_j)$ and $c_i\in \St^-_\fF(A_i;d_j)$ for all $i,j\in I$. Fix two representative systems of $\cA$: the future one $\bb=\{b_i\}_{i\in I}$ and the past one $\ba=\{a_i\}_{i\in I}$ such that $a_i\leq c_i$ and $d_i\leq b_i$ for all $i\in I$. The existence of such systems is guaranteed by Proposition \ref{p:ExistenceOfRepresentants}. Choose d-paths $\omega_i\in\vP(a_i,b_i)$.

\begin{prp}\label{p:FPEq}
	For all $i,j\in I$, both maps $\vP(\omega_i,b_j)$ and $\vP(a_i,\omega_j)$ are $\fF$--equivalences.
\end{prp}
\begin{proof}
	There exist $b'_j\in \St^+_\fF(\omega_i;A_j)_{\geq b_j}$ and $\eta_j\in \vP(b_j,b'_j)$. We have a commutative diagram
	\[
		\begin{diagram}
			\node{\vP(b_i,b_j)}
				\arrow{e,t}{\vP(\omega_i,b_j)}
				\arrow{s,l}{\simeq}
			\node{\vP(a_i,b_j)}
				\arrow{s,r}{\simeq}
		\\
			\node{\vP(b_i,b'_j)}
				\arrow{e,t}{\simeq}
			\node{\vP(a_i,b'_j)}
		\end{diagram}
	\]
	induced by the paths $\eta_j$ and $\omega_i$. All maps marked $\simeq$ are $\fF$--equivalences:
	\begin{itemize}
	\item{The left vertical map, since $b_j\in\St^+_\fF(b_i;A_j)$.}
	\item{The right vertical map, from Proposition \ref{p:PropertiesOfStabilizingPairs}, since $(c_i,d_j)$ is a stabilizing pair.}
	\item{The bottom map, since $b'_j\in\St^+_\fF(\omega_i;A_j)$.}
	\end{itemize}
	Thus, $\vP(\omega_i,b_j)$ is an $\fF$--equivalence. The dual argument shows that $\vP(a_i,\omega_j)$ is also an $\fF$--equivalence.
\end{proof}

\begin{prp}
	Assume that $\cA$ is a finite total $\fF$--component system on $X$. Then the categories $\vSP^+_{\fF,L}(X,\cA)$ and  $\vSP^-_{\fF,L}(X,\cA)$ are isomorphic.
\end{prp}
\begin{proof}
	This is an immediate consequence of Propositions \ref{p:NaturalEquiv} and \ref{p:FPEq}.
\end{proof}

If $X$ admits a coarsest future (resp. past, total) stable $\fF$--component system, we define \emph{the stable future (resp. past, total) $\fF$--component category} of $X$ as
\begin{equation}
	\vSP^+_{\fF,L}(X):= \vSP^+_{\fF,L}(X,\cM^+(X)),
\end{equation}
resp. $\vSP^-_{\fF,L}(X):= \vSP^-_{\fF,L}(X,\cM^-(X))$,
\begin{equation}
	\vSP^\pm_{\fF,L}(X):= \vSP^+_{\fF,L}(X,\cM^\pm(X))=\vSP^-_{\fF,L}(X,\cM^\pm(X)).
\end{equation}

\section{Examples and final remarks}

Below we provide several examples of calculations of coarsest stable component systems and  stable component categories.

\subsection*{Dubut's example and related examples}

\begin{exa}\label{x:Dubut}
	Consider three d-spaces:
	
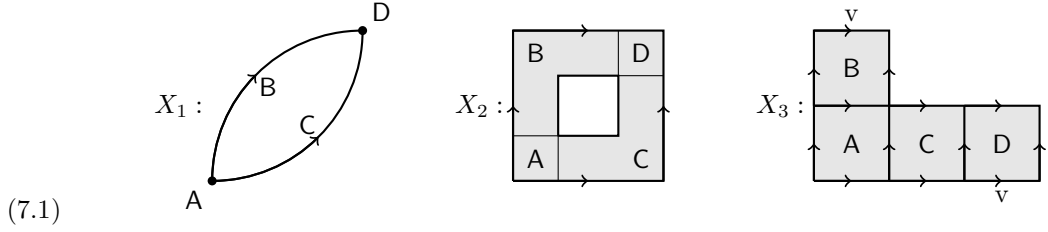
\begin{equation}
	\begin{tikzpicture}
		\fill (0,0) circle (0.06);
		\fill (2,2) circle (0.06);
		\draw[thick,->](0,0) arc (270:315:2 and 2);
		\draw[thick](0,0) arc (270:360:2 and 2);
		\draw[thick,->](0,0) arc (0:45:-2 and 2);
		\draw[thick](0,0) arc (0:90:-2 and 2);
		
		\fill[color=black!10!white] (4,0)--(4,2)--(6,2)--(6,0)--(4,0);
		\fill[color=white] (4.6,0.6)--(4.6,1.4)--(5.4,1.4)--(5.4,0.6)--(4.6,0.6);
		\draw[thick] (4,0)--(4,2)--(6,2)--(6,0)--(4,0);
		\draw[thick] (4.6,0.6)--(4.6,1.4)--(5.4,1.4)--(5.4,0.6)--(4.6,0.6);
		\draw[thick,->] (4,0)--(5,0);
		\draw[thick,->] (4,2)--(5,2);
		\draw[thick,->] (4,0)--(4,1);
		\draw[thick,->] (6,0)--(6,1);
		
		\fill[color=black!10!white] (8,0)--(11,0)--(11,1)--(9,1)--(9,2)--(8,2);
		\draw[thick] (8,0)--(11,0)--(11,1)--(8,1);
		\draw[thick] (8,0)--(8,2)--(9,2)--(9,0);
		\draw[thick] (10,0)--(10,1);
		\node[above] at (8.5,2) {v};
		\node[below] at (10.5,0) {v};
		\draw[thick,->] (8,0)--(8,0.5);
		\draw[thick,->] (9,0)--(9,0.5);
		\draw[thick,->] (10,0)--(10,0.5);
		\draw[thick,->] (11,0)--(11,0.5);
		\draw[thick,->] (8,1)--(8,1.5);
		\draw[thick,->] (9,1)--(9,1.5);
		\draw[thick,->] (8,0)--(8.5,0);
		\draw[thick,->] (9,0)--(9.5,0);
		\draw[thick,->] (10,0)--(10.5,0);
		\draw[thick,->] (8,1)--(8.5,1);
		\draw[thick,->] (9,1)--(9.5,1);
		\draw[thick,->] (10,1)--(10.5,1);
		\draw[thick,->] (8,2)--(8.5,2);
		
		\node[left] at (0,1) {$X_1:$};
		\node[left] at (4,1) {$X_2:$};
		\node[left] at (8,1) {$X_3:$};
		
		\node[below left] at (0,0) {$\mathsf{A}$};
		\node[above right] at (2,2) {$\mathsf{D}$};
		\node[below right] at (0.5,1.5) {$\mathsf{B}$};
		\node[above left] at (1.5,0.5) {$\mathsf{C}$};
		
		\node at (8.5,0.5) {$\mathsf{A}$};
		\node at (8.5,1.5) {$\mathsf{B}$};
		\node at (9.5,0.5) {$\mathsf{C}$};
		\node at (10.5,0.5) {$\mathsf{D}$};

		\node at (4.3,0.3) {$\mathsf{A}$};
		\node at (4.3,1.7) {$\mathsf{B}$};
		\node at (5.7,0.3) {$\mathsf{C}$};
		\node at (5.7,1.7) {$\mathsf{D}$};
		
		\draw (4,0.6)--(4.6,0.6)--(4.6,0);
		\draw (5.4,2)--(5.4,1.4)--(6,1.4);
	\end{tikzpicture}
	\end{equation}
	In the space $X_3$, the edges marked v are identified. All points in the boundaries between the areas $A,B,C,D$ belong either to $A$ or to $D$.  All spaces of d-paths between particular points are discrete up to homotopy; therefore, the coarsest $\fF$--component system does not depend on the choice of the system of equivalences $\fF$.	
	The coarsest systems of all these spaces are similar:
	\begin{itemize}
	\item{The coarsest future system is $\{A,B\cup C\cup D\}$,}
	\item{The coarsest past system is $\{A\cup B\cup C, D\}$,}
	\item{The coarsest total system is $\{A,B, C, D\}$.}
	\end{itemize}
	
	For every  $\mu\in\{+,-,\pm\}$, the categories $\vSP^\mu_{\fF,L}(X_i)$ are all isomorphic for $i=1,2,3$. For $\mu\in \{+,-\}$, $\vSP^\mu_{\fF,L}(X_i)$ has two objects $0,1$ and the morphisms given by $\vSP^\mu_{\fF,L}(X_i)(0,1)=L(S^0)$. The category $\vSP^\pm_{\fF,L}(X_i)$ has four objects $A,B,C,D$.
	For $X_1$ and $X_2$ these coincide with the component categories defined in \cite{GHComp2}.
\end{exa}

\subsection*{The boundary of the directed cube}

The boundary of the directed $n$--cube, $\partial\vI^n$, is the geometric realization of the pre-cubical set $\partial \square^n$, which is obtained from $\square^n$ by removing the sole $n$--cube $(*,\dots,*)$ (see Example \ref{x:nCube}). 	
	As a consequence of Theorem \ref{t:InteriorSystem}, $\mathcal{U}:=\{U_c\}_{c\in\coprod_k \partial\square^n[k]}$ is a total stable component system on $X$ (for any class of equivalences $\fF$). Our goal is to find the coarsest stable $\fF_k$--component systems for all $k$. To achieve this, we need to calculate the homotopy types of the d-path spaces between any two particular points of $\partial \vI^n$.

\begin{prp}\label{p:DPathSpacesInPartialCube}
	Let $\bx=(x_1,\dots,x_n)$ and $\by=(y_1,\dots,y_n)$ be points of $\partial\vI^n$. Let $\ba=(a_1,\dots,a_n)$ and $\bb=(b_1,\dots,b_n)$, $a_i,b_i\in \{0,1,*\}$, be the cubes of $\partial\square^n$ such that $\bx\in U_\ba$, $\by\in U_\bb$ (see Example \ref{x:UVCube}). Then:
	\begin{itemize}
	\item{If $x_i>y_i$ for some $i$, then $\vP(\partial\vI^n)_\bx^\by=\emptyset$.}
	\item{Otherwise, if $(a_i,b_i)\in \{(0,0),(0,*),(*,1),(1,1)\}$ for some $i$, then $\vP(\partial\vI^n)_\bx^\by$ is contractible.}
	\item{If neither of the conditions above is satisfied, then we have $(a_i,b_i)\in\{(0,1),(*,*)\}$ for all $i$. In this case, $\vP(\partial\vI^n)_\bx^\by$ is homotopy equivalent to the sphere $S^{r-2}$, where $r=\card(\{i\;|\; (a_i,b_i)=(0,1)\})$.}
	\end{itemize}
\end{prp}
\begin{proof}
	The first statement is obvious. Assume that $x_i\leq y_i$ for all $i$ and that there exists $i$ such that $(a_i,b_i)\in \{(0,0),(0,*),(*,1),(1,1)\}$; without loss of generality we may assume that $i=n$. Denote $\bx'=(x_1,\dots,x_{n-1})$, $\by'=(y_1,\dots,y_{n-1})$. For a path $\alpha\in\vP(\partial\vI^n)_\bx^\by$, let $\alpha=(\alpha',\alpha_n)$ be the presentation such that $\alpha'\in\vP(\vI^{n-1})_{\bx'}^{\by'}$, $\alpha_n\in\vP(\vI)_{x_n}^{y_n}$.
	 If $(a_n,b_n)=(0,0)$, then $x_n=y_n=0$ and
	\[
		\vP(\partial\vI^n)_\bx^\by=\vP(\vI^{n-1}\times \{0\})_{\bx}^{\by}\cong \vP(\vI^{n-1})_{\bx'}^{\by'}
	\]
	 is a contractible space. A similar argument works for $a_n=b_n=1$. If $(a_n,b_n)=(0,*)$, then we will show that the maps
	 \begin{align*}
	 	\vP(\partial\vI^n)_\bx^\by  \xtofrom[(\beta,\sigma_0)*(\sigma_{\by'},\gamma)\reflectbox{\tiny $\mapsto$}(\beta,\gamma):G]{F:\alpha\mapsto (\alpha',\alpha_n)}  \vP(\vI^{n-1})_{\bx'}^{\by'} \times \vP(\vI)_0^{y_n}
	 \end{align*}
	 are mutual homotopy inverses. A homotopy between $FG$ and $\id$ is given by point-wise convex combinations. The formula
	 \[
	 	H_s(\alpha)(t)=
	 	\begin{cases}
	 		(\alpha'((1+s)t),0) & \text{for $0\leq t\leq s(1+s)^{-1}$},\\
	 		(\alpha'((1+s)t),\alpha_n((1+s)t-s)) & \text{for $s(1+s)^{-1}\leq t\leq (1+s)^{-1}$},\\
	 		(\by',\alpha_n((1+s)t-s)) & \text{for $(1+s)^{-1}\leq t\leq 1$,}\\
	 	\end{cases}
	 \]
	$s,t\in[0,1]$, defines a homotopy between $\id_{\vP(\partial\vI^n)_\bx^\by}$ and $GF$ but we need to check that $H_s(\alpha)(t)\in\partial\vI^n$ for all $s,t,\alpha$. This is obvious for $t\leq s(1+s)^{-1}$.  Since $\alpha(t)\in\partial\vI^n$ for every $t$, then there exists $c_\alpha\in[0,1)$ such that $\alpha_n(t)=0$ for $t\in [0,c_\alpha]$ and $\alpha'(t)\in\partial\vI^{n-1}$ for $t\in [c_\alpha,1]$. In particular, $\by'\in\partial\vI^{n-1}$, which implies that $H_s(\alpha)(t)\in \partial\vI^n$ for $t\in[(1+s)^{-1}, 1]$. For $t\in [s(1+s)^{-1},(1+s)^{-1}]$, we have either $(1+s)t\geq c_\alpha$ or $((1+s)t-s)\leq (1+s)t\leq c_\alpha$. Thus, the homotopy $H_s$ is well-defined. As a consequence, $\vP(\partial\vI^n)_\bx^\by$ is contractible also in this case, as well as for $(a_n,b_n)=(*,1)$.
	 
	 Finally, assume that $(a_i,b_i)\in\{(0,1),(*,*)\}$ for all $i$, and denote $A=\{i\;|\; (a_i,b_i)=(0,1)\}$, $B=\{j\;|\; (a_j,b_j)=(*,*)\}$. Since
	 \[
	 	(\partial\vI^n)_{\geq \bx}^{\leq \by}= \prod_{j\in B} [x_j,y_j] \times \partial \vI^{A},
	 \]
	 we have
	 \[
	 	\vP(\partial\vI^n)_\bx^\by = \vP(\prod_{j\in B} [x_j,y_j] \times \partial \vI^{A})_\bx^\by \cong \prod_{j\in B} \vP([x_j,y_j])_{x_j}^{y_j} \times \vP(\partial\vI^A)_{(0,\dots,0)}^{(1,\dots,1)} \simeq \{*\} \times S^{\card(A)-2}.
	 \]
	 The last equivalence is proven in \cite[2.6.1]{RZ}.
\end{proof}

Now we are ready to calculate the coarsest future $\fF_k$--component systems on $\partial\vI^n$.

\begin{exa}
	 By Proposition \ref{p:DPathSpacesInPartialCube} we have
	\begin{equation}
		\vP_{\partial\vI^n}((0,\dots,0),(1,\dots,1))\simeq S^{n-2},
	\end{equation}
	and $\vP_{\partial\vI^n}(\bx,(1,\dots,1))$ is contractible for all other points $\bx\in \partial\vI^n$. Thus,
	\begin{enumerate}[\normalfont (a)]
	\item{
		If $k\geq n-2$, then  the coarsest future $\fF_k$--component system on $\partial\vI^n$ consists of two subsets $A_0=\{(0,0,\dots,0)\}$ and $A_1=\partial\vI^n\setminus \{0,0,\dots,0\}$. The future stable component category $\vSP^+_{\fF_k,L}(\partial\vI^n)$ is determined by $\vSP^+_{\fF_k,L}(\partial\vI^n)(0,1)=L(S^{n-2})$.
	}
	\item{
		If $k<n-2$, then $S^{n-2}$ is $\fF_k$--contractible. Thus, the singleton $\{\partial\vI^n\}$ is the coarsest total $\fF_k$--component system on $\partial\vI^n$, and the category $\vSP^+_{\fF_k,L}(\partial\vI^n)$ is trivial.
	}
	\end{enumerate}
	Since $\partial\vI^n$ is d-homeomorphic to its opposite d-space, similar facts hold for past component systems.
\end{exa}

\begin{prp}
	For $n\geq 3$ and $k\geq n-2$, $\mathcal{U}=\{U_c\}_{c\in \coprod_{k\geq 0} \partial\square^n[k]}$ (see {\normalfont(\ref{e:Uc})}) is the coarsest total stable component system on $\partial\vI^n$.
\end{prp}
\begin{proof}
 Obviously,  $\cM^\pm_{\fF_k}(\partial \vI^n)$ is finer than both $\cM^+_{\fF_k}(\partial \vI^n)$ and $\cM^-_{\fF_k}(\partial \vI^n)$, which implies that 
	\[
		\{(0,\dots,0)\}, \{(1,\dots,1)\}\in \cM^\pm_{\fF_k}(\partial \vI^n).
	\]
	By Theorem \ref{t:InteriorSystem}, $\mathcal{U}$ is a total stable component system on $\partial\vI^n$. Therefore, every element of $\cM:=\cM^\pm_{\fF_k}(\partial \vI^n)$ is a union of sets having the form $U_c$. Let $E_{(a_1,\dots,a_n)}$ denote the component of $\cM$ that contains $U_{(a_1,\dots,a_n)}$. Assume that $E_{(a_1,\dots,a_n)}=E_{(b_1,\dots,b_n)}$ for some $a_i,b_i\in \{0,1,*\}$; we will prove that $(a_1,\dots,a_n)=(b_1,\dots,b_n)$. 
	
	For $\varepsilon\in \{0,1\}$ and $i\in \{1,\dots,n\}$ denote $e(\varepsilon,i)=(\varepsilon,\dots,\varepsilon,*,\varepsilon,\dots,\varepsilon)\in \partial\square^n[1]$, where the only star appears at the $i$--th position. Since $(\partial\square^n)_{\geq x}= (U_{e(1,i)})_{\geq x}\cup \{(1,\dots,1)\}$ for every $x\in U_{e(1,i)}$, and $(1,\dots,1)\not\in E_{e(1,i)}$, then $U_{e(i,1)}$ and $E_{e(i,1)}$ are cofinal. Hence, by Proposition \ref{p:CommonStableTraces},
\[
	\vSP^+_{\fF_k}(U_{(a_1,\dots,a_n)},U_{e(1,i)})\simeq \vSP^+_{\fF_k}(E_{(a_1,\dots,a_n)},E_{e(1,i)})\simeq \vSP^+_{\fF_k}(U_{(b_1,\dots,b_n)},U_{e(1,i)}).
\]
It follows from Proposition \ref{p:DPathSpacesInPartialCube} that
\[
	\vSP^+_{\fF_k}(U_{(a_1,\dots,a_n)},U_{e(1,i)})\simeq \begin{cases}
		\emptyset & \text{for $a_i=1$,}\\
		\{*\} & \text{for $a_i\in\{0,*\}$, $(a_1,\dots,a_n)\neq e(0,i)$,}\\
		S^{n-3} & \text{for  $(a_1,\dots,a_n)=e(0,i)$,}
	\end{cases}
\]
which implies that $\{i\;|\; a_i=1\} = \{i\;|\; b_i=1\}$, since neither $S^{n-3}$ nor $\{*\}$ is $\fF_k$--equivalent to the empty set. Notice that the argument fails here for $n=2$. In a similar way, one can show that also $\{i\;|\; a_i=0\} = \{i\;|\; b_i=0\}$ and, therefore, $(a_1,\dots,a_n)=(b_1,\dots,b_n)$.
\end{proof}

\begin{exa} { \ }
	\begin{enumerate}[\normalfont (a)]
	\item{
		If $k\geq n-2$ and $n\geq 3$, then the coarsest total $\fF_k$--component system on $\partial\vI^n$ is the system $\cU=\{U_c\}_{c\in \coprod_k \partial \square^n[k]}$ constructed in Section \ref{s:Cubical}. The morphism objects of $\vSP^\pm_{\fF_k}(\partial\vI^n)$ are given by
		\[
			\vSP^\pm_{\fF_k,L}(\partial\vI^n)((a_1,\dots,a_n),(b_1,\dots,b_n))=\begin{cases}
				L(\emptyset) & \text{if $(a_i,b_i)\in \{(*,0), (1,0), (1,*)\}$ for some $i$,}\\
				L(S^{r-3}) & \text{if $\{(a_i,b_i)\}_{i=1}^n \subseteq \{(0,1),(*,*)\}$}\\
				L(\{*\}) & \text{otherwise,}
			\end{cases}
		\]
		which follows from Proposition \ref{p:DPathSpacesInPartialCube}. Notice that all compositions of morphisms are trivial, since no two morphisms objects of type $L(S^{r-3})$ are composable.
	}
	\item{
		If $k<n-2$ and $n\geq 3$, then $\cM^\pm_{\fF_k,L}(\partial \vI^n) =\{ \partial\vI^n\}$, since the spaces $\vP_{\partial\vI^n}(\bx,(1,\dots,1))$ and $\vP_{\partial\vI^n}((0,\dots,0),\bx)$ are $\fF_k$--contractible for all $\bx\in\partial\vI^n$. There exist pairs $\bx,\by\in\partial\vI^n$ such that $\vP_{\partial\vI^n}(\bx,\by)$ is \emph{not} $\fF_k$--contractible but these non-trivial d-path spaces are ``unstable": they trivialize when $\bx$ (resp. $\by$) is replaced by some smaller (resp. larger) point.
	}
	\item{
		If $n=2$, then $\partial\vI^2$ is d-homeomorphic to the space $X_1$ from Example \ref{x:Dubut}, and we have
\[
	\cM^\pm_{\fF}(\partial\square^2) =\{\{(0,0)\},\{(1,1)\},U_{(0,*)}\cup (0,1)\cup U_{(*,1)}, U_{(*,0)}\cup (1,0)\cup U_{(1,*)} \}
\]
	for every equivalence system $\fF$.
	}	
	\end{enumerate}
\end{exa}

\subsection*{Directed suspension}

\begin{exa}
	Let $X$ be a d-space with no loops. \emph{The directed suspension} of $X$ is the d-space
	\begin{equation}
		\vSigma( X) = (X \times \vI)/ (X\times \{0\}),\; (X\times \{1\}) \cong (X \times \vR)/ (X\times (-\infty,0]),\; (X\times [1,+\infty))
	\end{equation}
	with the quotient d-structure, i.e., the poorest one such that the projection $X\times\vI\to\vSigma(X)$ is a d-map. Denote the contracted points by $\bO$ and $\bI$ respectively. It is not very difficult to check that for $x,y\in X$, $s,t\in(0,1)$
	\begin{equation}
		\vP_{\vSigma(X)}((x,s),(y,t))\simeq \begin{cases}
			\vP_X(x,y) & \text{for $s\leq t$}\\
			\emptyset & \text{for $s>t$,}
		\end{cases}
	\end{equation}
	and that	
	\[
		\vP_{\vSigma(X)}((x,s),\bI)\simeq \vP_{\vSigma(X)}(\bO,(y,t)) \simeq \{*\}.
	\]
	Moreover, $\vP_{\vSigma(X)}(\bO,\bI)\simeq X$; this statement is not quite obvious and will be proven below.  Thus, if $X$ is not $\fF$--contractible, regarded as a topological space:
	\begin{enumerate}[\normalfont (a)]
	\item{$\cM^+_{\fF}(\vSigma(X))=\{\{\bO\},\vec\Sigma(X)\setminus\{\bO\}  \}$ and $\vSP^+_{\fF,L}(\vec\Sigma(X))$ has two objects, with $L(X)$ being the morphism object between them.}
	\item{If $X$ admits a coarsest total $\fF$--component system, then
\begin{equation}
	\cM^{\pm}_{\fF}(\vec\Sigma(X))=\{A\times (0,1)\;|\; A\in \cM^{\pm}_\fF(X) \}\cup \{ \{\bO\},\{\bI\}\},
\end{equation}	
	 and the category $\vSP^\pm_\fF(\vec\Sigma(X))$ can be obtained from $\vSP^\pm_{\fF}(X)$ by adding two objects $\bO$ and $\bI$ with the morphism objects
$\vSP^\pm_\fF(\vec\Sigma(X))(\bO,\bI)\simeq X$, and for $A\in\vSP^\pm_\fF(X)$
\begin{align*}
	\vSP^\pm_\fF(\vec\Sigma(X))(\bO,A) & \simeq \vSP^\pm_\fF(\vec\Sigma(X))(A,\bI) \simeq \{*\}\\
	\vSP^\pm_\fF(\vec\Sigma(X))(\bI,A) & \simeq \vSP^\pm_\fF(\vec\Sigma(X))(A,\bO) \simeq \emptyset.
\end{align*}	
	 }
	\end{enumerate}
\end{exa}

It remains to check that  $\vP_{\vSigma(X)}(\bO,\bI)$ is homotopy equivalent to $X$. This can be deduced from the calculations presented in \cite[Sections 8 \& 9]{WZ} but we will give a more elementary proof. 

\begin{prp}\label{p:Section}
	There exists a continuous map $q:\vP(\vI)_0^1\to (0,1)$ such that $0<\alpha(q(\alpha))<1$ for every $\alpha\in\vP(\vI)_0^1$.
\end{prp}
\begin{proof}
	Let $ \vP_{norm}(\vI^2)_{(0,0)}^{(1,1)}\subseteq \vP(\vI^2)_{(0,0)}^{(1,1)}$ denote the space of normalized d-paths, i.e., d-paths $\beta=(\beta_1,\beta_2)$ satisfying $\beta_1(t)+\beta_2(t)=t/2$ for all $t\in[0,1]$. In \cite[Section 2]{R-Trace}, Raussen constructed, in a much more general setting, the continuous map 
	\begin{equation*}
		N=norm\circ nat:\vP(\vI^2)_{(0,0)}^{(1,1)}\to \vP_{norm}(\vI^2)_{(0,0)}^{(1,1)}
	\end{equation*}
	such that every $\beta\in \vP(\vI^2)_{(0,0)}^{(1,1)}$ is a non-decreasing reparametrization of $N(\beta)$, i.e., there exists a non-decreasing continuous surjection $f_\beta:[0,1]\to[0,1]$ such that $\beta(t)=N(\beta)(f_\beta(t))$. Since $N(\beta)$ is injective, then $f_\beta$ is determined uniquely. Let $q$ be the composition
	\begin{equation*}
		\vP(\vI)_0^1 \xrightarrow{\alpha\mapsto\alpha\times\id_{\vI}}\vP(\vI^2)_{(0,0)}^{(1,1)} \xrightarrow{N} \vP_{norm}(\vI^2)_{(0,0)}^{(1,1)} \xrightarrow{\alpha\mapsto\alpha(1/2)} \{(x,y)\in\vI^2\;|\; x+y=1\} \xrightarrow{(x,y)\mapsto y} [0,1].
	\end{equation*}
	Notice that $(\alpha(t),t)\not\in \{(0,1),(1,0)\}$, which implies that $0<q(\alpha)<1$ for all $\alpha\in\vP(\vI)_0^1$. Choose $t\in (f_{(\alpha,\id_{\vI})})^{-1}(1/2)$; we have
	\begin{equation}
		(\alpha(t),t)=N(\alpha\times\id_{\vI})(1/2)=(1-q(\alpha),q(\alpha)),
	\end{equation}
	which shows that $t=q(\alpha)$ and  $\alpha(q(\alpha))=\alpha(t)=1-q(\alpha)\in (0,1)$.
\end{proof}

By $[x,t]$, $x\in X$, $t\in\R$ we denote the point of $\vSigma(X)$ represented by the pair $(x,t)$; we assume $[x,t]=\bO$ for $t\leq 0$ and $[x,t]=\bI$ for $t\geq 1$. Similarly, we will write $\alpha=[\alpha_X,\alpha_I]$ if the d-path $\alpha\in\vP(\vSigma(X))$ is represented by the pair $(\alpha_X\in\vP(X), \alpha_I\in\vP(\vI))$.

For $a\leq b\in\R$, $x\in X$, let $\beta^{[a,b]}_x$ be the d-path in $\vSigma(X)$ given by the formula $\beta^{[a,b]}_x(t)=[(1-t)a+tb, x]$.

\begin{prp}
	For every d-space $X$, the map $J:X\ni x\mapsto \beta^{[0,1]}_x \in \vP(\vSigma(X))_\bO^\bI$ is a homotopy equivalence.
\end{prp}
\begin{proof}
	We will show that the map
	\[
		S: \vP(\vSigma(X))_\bO^\bI\ni \alpha=[\alpha_X,\alpha_I]\mapsto \alpha_X(q(\alpha_I))\in X
	\]
	is a homotopy inverse of $J$. Notice that $S(\alpha)$ does not depend on the choice of $\alpha_X$ and $\alpha_I$. Clearly $S\circ J$ is the identity on $X$. We will construct a homotopy between $J\circ S$ and the identity on $\vP(\vSigma(X))_\bO^\bI$. The idea is to paste to a path $\alpha$ at $q(\alpha_I)$, where $q$ is the map from Proposition \ref{p:Section}, a longer and longer ``vertical" segment. 
	
	The formula $\sh(\alpha,h)(t)=[\alpha_I(t)+h,\alpha_X(t)]$ defines the continuous ``shift" map 
	\[
		\sh:\{(\alpha,h)\in \vP(\Sigma(X))\times \R\; |\; \text{ ($\alpha_I(1)<1$ and $h\leq 0$) or ($\alpha_I(0)>0$ and $h\geq 0$)} \} \to  \vP(\Sigma(X)).
	\]
	Let $\alpha|_{[a,b]}$ denote the path $\alpha|_{[a,b]}(t)=\alpha((1-t)a+tb)$. The formula
	\[
		H_s(\alpha)=\sh(\alpha|_{[0,q(\alpha_I)]},-s\alpha_I(q(\alpha_I)))*\beta^{[a(s,\alpha),b(s,\alpha)]}_{\alpha_X(q(\alpha_I))}*\sh(\alpha|_{[q(\alpha_I),1]},s(1-\alpha_I(q(\alpha_I)))),
	\]
	where
	\begin{align*}
		a(s,\alpha)&=(1-s)\alpha_I(q(\alpha_I))\\
		b(s,\alpha)&=\alpha_I(q(\alpha_I)) + s(1-\alpha_I(q(\alpha_I))),
	\end{align*}
	defines the homotopy between self-maps of $\vP(\vSigma(X))_\bO^\bI$; notice that the condition $0<\alpha_I(q(\alpha_I))<1$ guarantees that $\alpha_X(q(\alpha_I))$ is well-defined, and that the arguments of the functions $\sh$ lie in the domain. We have
	\begin{align*}
		H_0(\alpha)&=\alpha|_{[0,q(\alpha_I)]}*\sigma_{\alpha(q(\alpha_I))}* \alpha|_{[q(\alpha_I),1]},\\
		H_1(\alpha)&=\sigma_\bO * \beta^{[0,1]}_{\alpha_X(q(\alpha_I))} * \sigma_\bI,
	\end{align*}
	and it is elementary to check that $H_0\sim \id_{ \vP(\vSigma(X))_\bO^\bI}$, $H_1\sim J\circ S$.
\end{proof}

\subsection*{Future work}
The stable component categories constitute a family of invariants of directed spaces with no loops. Still, they suffers some limitations; the author hopes that the following question can be answered.

\begin{itemize}
\item{Functoriality. Let $X$ and $Y$ be d-spaces with no loops that admit a coarsest (future/past/total) $\fF$--component system. Assume given a d-map  $f:X\to Y$. Is it possible to define an induced functor $\vSP^\mu_{\fF,L}(f):\vSP^\mu_{\fF,L}(X)\to \vSP^\mu_{\fF,L}(Y)$ in a functorial way?}
\item{Loops. The construction above essentially uses the fact that d-spaces under consideration have no loops. However, in many cases, for example for pre-cubical complexes, loops can be avoided by passing to universal, or loop-length coverings. These coverings do not, in general, admit finite stable component systems but the author hopes that it is possible to prove they admit the coarsest component systems, which can be used to define component categories also in this case.}
\item{Connections with \cite{DGG}. The stable component categories $\vSP^\mu_{\fH_R,H_*}(X)$ seem to be related to natural homology, introduced by Dubut, Goubault and Goubault-Larrecq. It would be interesting to investigate connections between these two invariants.}
\end{itemize}

\end{document}